\documentclass{amsart}

\usepackage{latexsym,enumerate}
\usepackage{amsmath,amsthm,amsopn,amstext,amscd,amsfonts,amssymb}
\usepackage[ansinew]{inputenc}
\usepackage{verbatim}
\usepackage{graphicx}
\usepackage{pstricks}
\usepackage{wasysym}
\usepackage[all]{xy}
\usepackage{epsfig} % Para grï¿½ficos
\usepackage{graphicx,psfrag} % Mejor ï¿½ste para    grï¿½ficos: es mï¿½s moderno
\usepackage{subfigure}
\usepackage{pstricks,pst-node}
\usepackage{multicol}
\usepackage{color}
\usepackage{colortbl}

\newcommand{\RR}{\mathbb{R}}

\newtheorem{theorem}{Theorem}[section]
\newtheorem{lemma}[theorem]{Lemma}
\newtheorem{proposition}[theorem]{Proposition}
\newtheorem{corollary}[theorem]{Corollary}
\newtheorem{definition}[theorem]{Definition}

\newtheorem{remark}[theorem]{Remark}

\DeclareMathOperator{\diam}{diam}

\newcommand{\spb}[1]{\smallskip}
\newcommand{\mpb}[1]{\medskip}
\newcommand{\bpb}[1]{\bigskip}

\newcommand{\p}{\partial}

\renewcommand{\a}{\alpha}
\renewcommand{\b}{\beta}
\newcommand{\e}{\varepsilon}
\renewcommand{\d}{\delta}

\newcommand{\g}{\gamma}
\newcommand{\G}{\Gamma}

\newcommand{\s}{\sigma}
\begin{document}
\DeclareGraphicsExtensions{.jpg,.pdf,.mps,.png}

\title{Hyperbolicity of direct products of graphs}
\author[Walter Carballosa]{Walter Carballosa$^{(1)}$}
\address{National Council of Science and Technology (CONACYT) $\&$ Autonomous University of Zacatecas,
Paseo la Bufa, int. Calzada Solidaridad, 98060 Zacatecas, ZAC, M\'exico}
\email{waltercarb@gmail.com}
\thanks{$^{(1)}$  Supported in part by two grants from Ministerio de Econom{\'\i}a y Competititvidad (MTM2013-46374-P and MTM2015-69323-REDT), Spain.}

\author[Amauris de la Cruz]{Amauris de la Cruz$^{(1)}$}
\address{Departamento de Matem\'aticas, Universidad Carlos III de Madrid,
Avenida de la Universidad 30, 28911 Legan\'es, Madrid, Spain}
\email{alcruz@math.uc3m.es}

\author[Alvaro Mart\'{\i}nez-P\'erez]{Alvaro Mart\'{\i}nez-P\'erez$^{(2)}$}
\address{ Facultad CC. Sociales de Talavera,
Avda. Real Fábrica de Seda, s/n. 45600 Talavera de la Reina, Toledo, Spain}
\email{alvaro.martinezperez@uclm.es}
\thanks{$^{(2)}$ Supported in part by a grant
from Ministerio de Econom{\'\i}a y Competitividad (MTM 2012-30719), Spain.
}

\author[Jos\'e M. Rodr{\'\i}guez]{Jos\'e M. Rodr{\'\i}guez$^{(1)(3)}$}
\address{Departamento de Matem\'aticas, Universidad Carlos III de Madrid,
Avenida de la Universidad 30, 28911 Legan\'es, Madrid, Spain}
\email{jomaro@math.uc3m.es}
\thanks{$^{(3)}$ Supported in part by a grant from CONACYT (FOMIX-CONACyT-UAGro 249818), M\'exico.}

\date{\today}

\maketitle{}

\begin{abstract}
If $X$ is a geodesic metric space and $x_1,x_2,x_3\in X$, a {\it
geodesic triangle} $T=\{x_1,x_2,x_3\}$ is the union of the three
geodesics $[x_1x_2]$, $[x_2x_3]$ and $[x_3x_1]$ in $X$. The space
$X$ is $\d$-\emph{hyperbolic} $($in the Gromov sense$)$ if any side
of $T$ is contained in a $\d$-neighborhood of the union of the two
other sides, for every geodesic triangle $T$ in $X$.
If $X$ is hyperbolic, we denote by
$\d(X)$ the sharp hyperbolicity constant of $X$, i.e.,
$\d(X)=\inf\{\d\ge 0: \, X \, \text{ is $\d$-hyperbolic}\,\}.$
Some previous works characterize the hyperbolic product graphs (for the Cartesian, strong, join, corona and lexicographic products) in terms of properties of the factor graphs.
However, the problem with the direct product is more complicated.
In this paper, we prove that if the direct product $G_1\times G_2$ is hyperbolic, then one factor is hyperbolic and the other one is bounded.
Also, we prove that this necessary condition is, in fact, a characterization in many cases.
In other cases, we find characterizations which are not so simple.
Furthermore, we obtain formulae or good bounds for the hyperbolicity constant of the direct product of some important graphs.
\end{abstract}



\section{Introduction}

The different kinds of products of graphs are an important research topic.
Some large graphs are composed from some existing smaller ones by using several products of graphs,
and many properties of such large graphs are strongly associated with that of the corresponding smaller ones.
In particular, given two graphs $G_1,G_2$, the \emph{direct product} $G_1 \times G_2$ is the graph with the vertex set $V (G_1 \times G_2)$,
and such that two vertices $(u_1,v_1)$ and $(u_2,v_2)$ of $G_1\times G_2$ are adjacent if $[u_1,u_2]\in E(G_1)$ and $[v_1,v_2]\in E(G_2)$.
The direct product is clearly commutative and associative.
Weichsel observed that $G_1 \times G_2$
is connected if and only if $G_1$ and $G_2$ are connected and $G_1$ or $G_2$ is not a bipartite graph \cite{W}.
Many different properties of direct product of graphs have been studied (sometimes with various different names,
such as cardinal product, tensor product, Kronecker product, categorical product,
conjunction,...). The study includes structural results \cite{BH, BIKB, H1, IR, IS, JK}, hamiltonian properties \cite{BP, KK}, and above all the well-known Hedetniemi's
conjecture on chromatic number of direct product of two graphs (see \cite{IK00} and \cite{Z}).
Open problems in the area suggest that a deeper structural understanding of
this product would be welcome.

Hyperbolic spaces play an important role in geometric
group theory and in the geometry of negatively curved
spaces (see \cite{ABCD, GH, G1}).
The concept of Gromov hyperbolicity grasps the essence of negatively curved
spaces like the classical hyperbolic space, simply connected Riemannian manifolds of
negative sectional curvature bounded away from $0$, and of discrete spaces like trees
and the Cayley graphs of many finitely generated groups. It is remarkable
that a simple concept leads to such a rich
general theory (see \cite{ABCD, GH, G1}).

The first works on Gromov hyperbolic spaces deal with
finitely generated groups (see \cite{G1}).
Initially, Gromov spaces were applied to the study of automatic groups in the science of computation
(see, e.g., \cite{O}); indeed, hyperbolic groups are strongly geodesically automatic, i.e., there is an automatic structure on the group \cite{Cha}.
The concept of hyperbolicity appears also in discrete mathematics, algorithms
and networking. For example, it has been shown empirically
in \cite{ShTa} that the internet topology embeds with better accuracy
into a hyperbolic space than into a Euclidean space
of comparable dimension (formal proofs that the distortion is related to the hyperbolicity can be found in \cite{VeSu});
furthermore, it is evidenced that many real networks are hyperbolic (see, e.g., \cite{AAD,ASM,CoCoLa,KPKVB,MoSoVi}).
%A few algorithmic problems in hyperbolic spaces and hyperbolic graphs have been considered in recent papers (see \cite{ChEs, Epp, GaLy, Kra}).
Another important application of these spaces is the study of the spread of viruses through the internet (see \cite{K21,K22}).
Furthermore, hyperbolic spaces are useful in secure transmission of information on the
network (see \cite{%K27,
K21,K22});
%,NS});
also to traffic flow and effective resistance of networks \cite {CDV,GJ,LiTu}.
The hyperbolicity has also been used extensively in the context of random graphs (see, e.g., \cite{Sha1,Sha2,Sha3}).

In \cite{T} it was proved the
equivalence of the hyperbolicity of many negatively curved surfaces
and the hyperbolicity of a graph related to it; hence, it is useful
to know hyperbolicity criteria for graphs from a geometrical viewpoint.
Hence, the study of Gromov hyperbolic graphs is a subject of increasing interest; see, e.g.,
\cite{AAD,ASM,BRS,
%BRST,
BRSV2,BKM,CoCoLa,
%CDR,
K21,K22,MP,MRSV2,RSVV,Sha1,Sha2,Sha3,T,VeSu,WZ} and the references therein.

We say that a curve $\g:[a,b] \rightarrow X$ in a metric space $X$ is a
\emph{geodesic} if we have $L(\g|_{[t,s]})=d(\g(t),\g(s))=|t-s|$ for every $s,t\in [a,b]$,
where $L$ and $d$ denote length and distance, respectively, and $\g|_{[t,s]}$ is the restriction of the curve $\g$ to the interval $[t,s]$
(then $\gamma$ is equipped with an arc-length parametrization).
The metric space $X$ is said \emph{geodesic} if for every couple of points in
$X$ there exists a geodesic joining them; we denote by $[xy]$
any geodesic joining $x$ and $y$; this notation is ambiguous, since in general we do not have uniqueness of
geodesics, but it is very convenient.
Consequently, any geodesic metric space is connected.
If the metric space $X$ is
a graph, then the edge joining the vertices $u$ and $v$ will be denoted by $[u,v]$.

In order to consider a graph $G$ as a geodesic metric space, identify (by an isometry)
any edge $[u,v]\in E(G)$ with the interval $[0,1]$ in the real line;
then the edge $[u,v]$ (considered as a graph with just one edge)
is isometric to the interval $[0,1]$.
Thus, the points in $G$ are the vertices and, also, the points in the interior
of any edge of $G$.
In this way, any connected graph $G$ has a natural distance
defined on its points, induced by taking shortest paths in $G$,
and we can see $G$ as a metric graph.
If $x,y$ are in different connected components of $G$, we define $d_G(x,y)=\infty$.

Throughout this paper, $G=(V,E)=(V(G),E(G))$ denotes a connected simple (without loops and multiple edges) graph such that every edge has length $1$
and $V\neq \emptyset$.
These properties guarantee that $G$ is a geodesic metric space.
Note that to exclude multiple
edges and loops is not an important loss of generality, since
\cite[Theorems 8 and 10]{BRSV2} reduce the problem of computing
the hyperbolicity constant of graphs with multiple edges and/or
loops to the study of simple graphs.

If $X$ is a geodesic metric space and $x_1,x_2,x_3\in X$, the union
of three geodesics $[x_1 x_2]$, $[x_2 x_3]$ and $[x_3 x_1]$ is a
\emph{geodesic triangle} that will be denoted by $T=\{x_1,x_2,x_3\}$
and we will say that $x_1,x_2$ and $x_3$ are the vertices of $T$; it
is usual to write also $T=\{[x_1x_2], [x_2x_3], [x_3x_1]\}$. We say
that $T$ is $\d$-{\it thin} if any side of $T$ is contained in the
$\d$-neighborhood of the union of the two other sides. We denote by
$\d(T)$ the sharp thin constant of $T$, i.e., $ \d(T):=\inf\{\d\ge 0:
\, T \, \text{ is $\d$-thin}\,\}. $ The space $X$ is
$\d$-\emph{hyperbolic} $($or satisfies the {\it Rips condition} with
constant $\d)$ if every geodesic triangle in $X$ is $\d$-thin. We
denote by $\d(X)$ the sharp hyperbolicity constant of $X$, i.e.,
$\d(X):=\sup\{\d(T): \, T \, \text{ is a geodesic triangle in
}\,X\,\}.$
We say that $X$ is
\emph{hyperbolic} if $X$ is $\d$-hyperbolic for some $\d \ge 0$; then $X$ is hyperbolic if and only if
$ \d(X)<\infty$.
If we have a triangle with two
identical vertices, we call it a ``bigon". Obviously, every bigon in
a $\d$-hyperbolic space is $\d$-thin.
If $X$ has connected components $\{X_i\}_{i\in I}$, then we define $\d(X):=\sup_{i\in I} \d(X_i)$, and we say that $X$ is hyperbolic if $\d(X)<\infty$.

In the classical references on this subject (see, \emph{e.g.}, \cite{ABCD,BHB,GH})
appear several different definitions of Gromov hyperbolicity, which are equivalent in the sense
that if $X$ is $\d$-hyperbolic with respect to one definition,
then it is $\d'$-hyperbolic with respect to another definition (for some $\d'$ related to $\d$).
The definition that we have chosen has a deep geometric meaning (see, \emph{e.g.}, \cite{GH}).

We want to remark that the main examples of hyperbolic graphs are the trees.
In fact, the hyperbolicity constant of a geodesic metric space can be viewed as a measure of
how ``tree-like'' the space is, since those spaces $X$ with $\delta(X) = 0$ are precisely the metric trees.
This is an interesting subject since, in
many applications, one finds that the borderline between tractable and intractable
cases may be the tree-like degree of the structure to be dealt with
(see, e.g., \cite{CYY}).

For a finite graph with $n$ vertices it is possible to compute $\d(G)$ in time $O(n^{3.69})$ \cite{FIV} (this is improved in \cite{CoCoLa,CD}).
Given a Cayley graph (of a presentation with solvable word problem) there is an algorithm which allows to decide if it is hyperbolic \cite{Pap}.
However, deciding whether or not a general infinite graph is hyperbolic is usually very difficult.
Thus, a way to approach the problem is to study hyperbolicity for particular types of graphs.
In this line, many researches have studied the hyperbolicity of several classes of graphs: chordal graphs \cite{BCRS,BKM,MP,WZ},
%median graphs (\cite{Si}), line graphs (\cite{CRS,CRSV,CCDL}), cubic graphs (\cite{PeRSV}), complement graphs (\cite{BRST}), regular graphs (\cite{HRSTV}), planar graphs (\cite{CPRS1,PRSV}), periodic graphs (\cite{CGPR,CGPR1}), short graphs (\cite{R}), minor graphs (\cite{CRRS}), Mycielskian graphs (\cite{GPPR}), geometric graphs (\cite{CCDL,RS1}), circulant graphs (\cite{HRS,RS}),
vertex-symmetric graphs \cite{CaFu}, bipartite and intersection graphs \cite{CoDu}, bridged graphs \cite{K56}, expanders \cite{LiTu}
%graphs with small hyperbolicity constant (\cite{BRS2})
and some products of graphs: Cartesian product \cite{MRSV2}, strong product \cite{CCCR}, corona and join product \cite{CRS1}.

In this paper we characterize in many cases the hyperbolic direct product of graphs.
Here the situation is more complex than with the Cartesian or the strong product, which is in part due to the
facts that the direct product of two bipartite graphs is already disconnected and that the formula for the distance in $G_1\times G_2$
is more complicated that in the case of other products of graphs.
Theorem \ref{t:neccon} proves that if $G_1\times G_2$ is hyperbolic, then one factor is hyperbolic and the other one is bounded.
Also, we prove that this necessary condition is, in fact, a characterization in many cases.
If $G_1$ is a hyperbolic graph and $G_2$ is a bounded graph, then we prove that $G_1\times G_2$ is hyperbolic when
$G_2$ has some odd cycle (Theorem \ref{t:G2odd}) or $G_1$ and $G_2$ do not have odd cycles (Theorem \ref{t:no-odd}).
Otherwise, the characterization is a more difficult task;
if $G_1$ has some odd cycle and $G_2$ does not have odd cycles, Theorems \ref{t:P2} and \ref{t:dense}
provide sufficient conditions for non-hyperbolicity and hyperbolicity, respectively;
besides, Theorems \ref{t:char} and Corollary \ref{c:char}
characterize the hyperbolicity of $G_1\times G_2$ under some additional conditions.
Furthermore, we obtain formulae or good bounds for the hyperbolicity constant of the direct product of some important graphs
(in particular, Theorem \ref{t:bipartite} provides the precise value of the hyperbolicity constant for many direct products of bipartite graphs).

%------------------------------------SECTION 2

\
\section{Hyperbolic direct products}

In order to study the hyperbolicity constant of the direct product of two graphs $G_1\times G_2$, we will need bounds for the distance between two arbitrary points.
We will use the definition given in \cite{HIK}.

\begin{definition}\label{def:DP}
Let $G_1=(V(G_1),E(G_1))$ and $G_2=(V(G_2),E(G_2))$ be two graphs. The \emph{direct product} $G_1\times G_2$ of $G_1$ and $G_2$ has $V(G_1) \times V(G_2)$ as vertex set, so that two distinct vertices $(u_1,v_1)$ and $(u_2,v_2)$ of $G_1\times G_2$ are adjacent if $[u_1,u_2]\in E(G_1)$ and $[v_1,v_2]\in E(G_2)$.
\end{definition}

If $G_1$ and $G_2$ are isomorphic, we write $G_1 \simeq G_2$. It is clear that if $G_1\simeq G_2$, then $\d(G_1)=\d(G_2)$.

From the definition, it follows that the direct product of two graphs is commutative, i.e., $G_1\times G_2\simeq G_2\times G_1$.
Hence, the conclusion of every result in this paper with some ``non-symmetric" hypothesis also holds if we change the roles of $G_1$ and $G_2$ (see, e.g., Theorems \ref{t:G2odd}, \ref{t:no-odd}, \ref{t:P2}, \ref{t:dense}
and \ref{t:char} and Corollary \ref{c:char}).

In what follows we denote by $\pi_i$ the projection $\pi_i:V(G_1\times G_2)\rightarrow V(G_i)$ for $i\in\{1,2\}$.
Note that, in fact, this projection is well defined as a map $\pi_i: G_1\times G_2 \rightarrow  G_i$ for $i\in\{1,2\}$.

\smallskip

We collect some previous results of \cite{HIK}, which will be useful.
If $G$ is a graph and $u, u'\in V(G)$, then by a $u, u'$-\emph{walk} in $G$ we mean a path joining $u$ and $u'$ where repeating vertices is allowed.

\begin{proposition}\cite[Proposition 5.7]{HIK}
\label{p:distvert}
Suppose $(u,v)$ and $(u',v')$ are vertices of the direct product $G_1\times G_2$, and $n$ is an integer for which $G_1$ has a $u, u'$-walk of length $n$ and $G_2$ has a $v, v'$-walk of length $n$. Then $G_1\times G_2$ has a walk of length $n$ from $(u,v)$ to $(u',v')$. The smallest such $n$ (if it exists) equals $d_{G_1\times G_2}((u,v),(u',v'))$. If no such $n$ exists, then $d_{G_1\times G_2}((u,v),(u',v'))= \infty$.
\end{proposition}

\begin{proposition}\cite[Proposition 5.8]{HIK} \label{p:distvert1}
Suppose $x$ and $y$ are vertices of $G_1\times G_2$. Then
$$
d_{G_1\times G_2}(x, y) = \min\big\{n \in \mathbb{N}\,| \text{ each factor }  G_i  \text{ has a $\pi_{i}(x), \pi_{i}(y)$-walk of length $n$ for }  i=1,2\big\},
$$
where it is understood that $d_{G_1 \times G_2}(x,y)=\infty$ if no such $n$ exists.
\end{proposition}

\begin{definition}\label{def:diam}
 The diameter of the vertices of the graph $G$, denoted by $\diam V(G)$, is defined as
    $$\diam V(G):= \sup \{d_{G}(u,v): u,v\in V(G)\},$$
 and the diameter of the graph $G$, denoted by $\diam G$, is defined as
$$\diam G:= \sup \{d_{G}(x,y): x,y\in G\}.$$
\end{definition}

\begin{corollary} \label{c:lower}
We have for every $(u,v),(u',v') \in V(G_1\times G_2)$
$$
d_{G_1\times G_2}((u,v),(u',v')) \ge \max\big\{ d_{G_1}(u,u'),\, d_{G_2}(v,v') \big\}
$$
and, consequently,
$$
\diam V(G_1\times G_2) \ge \max\big\{ \diam V(G_1),\, \diam V(G_2) \big\}.
$$
\medskip

Furthermore, if $d_{G_1}(u,u')$ and $d_{G_2}(v,v')$ have the same parity, then
$$
d_{G_1\times G_2}((u,v),(u',v')) = \max\big\{ d_{G_1}(u,u'),\, d_{G_2}(v,v') \big\}
$$
and, consequently,
$$
\diam V(G_1\times G_2) = \max\big\{ \diam V(G_1),\, \diam V(G_2) \big\}.
$$
\end{corollary}

In this paper by trivial graph we mean a graph having just a single vertex.

The following theorem, first proved by Weichsel in $1962$, characterizes connectedness in direct products of two factors.
As usual, by \emph{cycle} we mean a simple closed curve, i.e., a path with different vertices,
unless the last one, which is equal to the first vertex.

\begin{theorem}\cite[Theorem 5.9]{HIK} \label{t:conect}
Suppose $G_1$ and $G_2$ are connected non-trivial graphs. If at least one of $G_1$ or $G_2$ has an odd cycle,
 then $G_1\times G_2$ is connected. If both $G_1$ and $G_2$ are bipartite, then $G_1\times G_2$ has exactly two connected components.
\end{theorem}

\begin{corollary}\cite[Corollary 5.10]{HIK}
\label{c:conect}
A direct product of connected non-trivial graphs is connected if and only if at most one of the factors is bipartite.
In fact, the product has $2^{\max\{k,1\}-1}$ connected components, where $k$ is the number of bipartite factors.
\end{corollary}

\begin{proposition}\label{p:unbounded} Let $G_1$ and $G_2$ be two unbounded graphs. Then $G_1\times G_2$ is not hyperbolic.
\end{proposition}

\begin{proof} Since $G_1$ and $G_2$ are unbounded graphs, for each positive integer $n$ there exist two geodesic paths $P_1:=[w_1,w_2]\cup [w_2,w_3]\cup \cdots \cup [w_{n-1},w_n]$ in $G_1$ and $P_2:=[v_1,v_2]\cup [v_2,v_3]\cup \cdots \cup [v_{n-1},v_n]$ in $G_2$. If $n$ is odd, then we can consider the geodesic triangle $T$ in $G_1\times G_2$ defined by the following geodesics:
$$
\begin{aligned}
\gamma_1 & := [(w_1,v_2),(w_2,v_1)] \cup [(w_2,v_1),(w_3,v_2)] \cup  [(w_3,v_2),(w_4,v_1)] \cup  \cdots \cup [(w_{n-1},v_1),(w_n,v_2)],
\\
\gamma_2 & := [(w_1,v_2),(w_2,v_3)] \cup [(w_2,v_3),(w_1,v_4)] \cup  [(w_1,v_4),(w_2,v_5)] \cup \cdots \cup [(w_{1},v_{n-1}),(w_2,v_n)],
\\
\gamma_3 & := [(w_2,v_n),(w_3,v_{n-1})] \cup [(w_3,v_{n-1}),(w_4,v_{n-2})] \cup  [(w_4,v_{n-2}),(w_5,v_{n-3})] \cup \cdots \cup [(w_{n-1},v_{3}),(w_n,v_2)],
\end{aligned}
$$

Corollary \ref{c:lower} gives that $\gamma_1, \gamma_2, \gamma_3$ are geodesics.

Let $m:=\frac{n+1}{2}$ and consider the vertex $(w_{m},v_{m+1})$ in $\gamma_3$. For every vertex $(w_i,v_j)$ in $\gamma_1$, $j\in \{1,2\},$ we have $d_{G_1\times G_2}((w_m,v_{m+1}),(w_i,v_j))\ge d_{G_2}(v_{m+1},v_j)\ge m+1-2=\frac{n-1}{2}$ by Corollary \ref{c:lower}. We have for every vertex $(w_i,v_j)$ in $\gamma_2$, $i\in \{1,2\}$, by Corollary \ref{c:lower}, $d_{G_1\times G_2}((w_m,v_{m+1}),(w_i,v_j))\ge d_{G_1}(w_m,w_i)\ge m-2=\frac{n-3}{2}$. Hence, $d_{G_1\times G_2}\big((w_{m},v_{m+1}),\gamma_1\cup \gamma_2\big)\geq \frac{n-3}{2}$ and $\d(G_1\times G_2)\ge \d(T)\ge\frac{n-3}{2}$. Since $n$ is arbitrarily large, $G_1\times G_2$ is not hyperbolic.
\end{proof}

Let $(X,d_X)$ and $(Y,d_Y)$  be two metric spaces. A map $f: X\longrightarrow Y$ is said to be
an $(\alpha, \beta)$-\emph{quasi-isometric embedding}, with constants $\alpha\geq 1,\, \beta\geq 0$ if, for every $x, y\in X$:
$$
\alpha^{-1}d_X(x,y)-\beta\leq d_Y(f(x),f(y))\leq \alpha d_X(x,y)+\beta.
$$
The function $f$ is $\varepsilon$-\emph{full} if
for each $y \in Y$ there exists $x\in X$ with $d_Y(f(x),y)\leq \varepsilon$.

A map $f: X\longrightarrow Y$ is said to be
a \emph{quasi-isometry}, if there exist constants $\alpha\geq 1,\, \beta,\varepsilon \geq 0$ such that $f$ is an $\varepsilon$-full
$(\alpha, \beta)$-quasi-isometric embedding.

Two metric spaces $X$ and $Y$ are \emph{quasi-isometric} if there exists
a quasi-isometry $f:X\longrightarrow Y$.
One can check that to be quasi-isometric is an equivalence relation.
An $(\alpha, \beta)$-\emph{quasi-geodesic} in $X$ is an
$(\alpha, \beta)$-quasi-isometric embedding between an interval of $\RR$ and $X$.

A fundamental property of hyperbolic spaces is the following (see, e.g., \cite[p.88]{GH}):

\begin{theorem}[Invariance of hyperbolicity]\label{invarianza}
Let $f:X\longrightarrow Y$ be an $(\alpha,\beta)$-quasi-isometric embedding between the geodesic
metric spaces $X$ and $Y$.
If $Y$ is hyperbolic, then $X$ is hyperbolic.

Besides, if $f$ is $\e$-full for some $\e\ge0$ (a quasi-isometry), then $X$ is
hyperbolic if and only if $\,Y$ is hyperbolic.
\end{theorem}

\begin{lemma} \label{l:vertices}
Consider two graphs $G_1$ and $G_2$.
If $f: V(G_1) \longrightarrow V(G_2)$ is an $(\alpha,\beta)$-quasi-isometric embedding, then
there exists an $(\alpha,\alpha+\beta)$-quasi-isometric embedding $g: G_1 \longrightarrow G_2$ with $g=f$ on $V(G_1)$.
Furthermore, if $f$ is $\e$-full, then $g$ is $(\e+\frac{1}{2})$-full.
\end{lemma}

\begin{proof}
For each $x\in G_1$, let us choose a closest point $v_x \in V(G_1)$ from $x$, and define $g(x):=f(v_x)$.
Note that $v_x=x$ if $x \in V(G_1)$ and so $g=f$ on $V(G_1)$.
Given $x,y\in G_1,$ we have
$$
\begin{aligned}
d_{G_2}(g(x),g(y)) & = d_{G_2}(f(v_x),f(v_y)) \leq \alpha d_{G_1}(v_x,v_y)+\beta \le \alpha \big(d_{G_1}(x,y) + 1\big) +\beta ,
\\
d_{G_2}(g(x),g(y)) & = d_{G_2}(f(v_x),f(v_y)) \geq \alpha^{-1} d_{G_1}(v_x,v_y)-\beta \ge \alpha^{-1} \big(d_{G_1}(x,y) - 1\big) -\beta ,
\end{aligned}
$$
and $g$ is an $(\alpha,\alpha+\beta)$-quasi-isometric embedding, since $\alpha \ge 1 \ge \alpha^{-1}$.

Furthermore, if $f$ is $\e$-full, then $g$ is $(\e+\frac{1}{2})$-full since $g(G_1)=f(V(G_1))$.
\end{proof}

Given a graph $G$, let $g_I(G)$ denote the \emph{odd girth} of $G$, this is, the length of the shortest odd cycle in $G$.

\begin{theorem}\label{t:G2odd} Let $G_1$ be a graph and $G_2$ be a non-trivial bounded graph with some odd cycle. Then, $G_1 \times G_2$ is hyperbolic if and only if $G_1$ is hyperbolic.
\end{theorem}

\begin{proof} Let $v_0\in V(G_2)$ such that $v_0$ is contained in an odd cycle $C$ with $L(C)=g_I(G_2)$. Consider the map $i:V(G_1) \to V(G_1\times G_2)$ such that $i(w):=(w,v_0)$ for every $w\in V(G_1)$.

By Corollary \ref{c:lower}, for any pair of vertices $w_1,w_2\in V(G_1)$,
$d_{G_1}(w_1,w_2)\leq d_{G_1\times G_2}\big((w_1,v_0),(w_2,v_0) \big)$. Also, Proposition \ref{p:distvert1} gives the following.

If a geodesic joining $w_1$ and $w_2$ has even length, then
$$d_{G_1\times G_2}\big((w_1,v_0),(w_2,v_0) \big)=d_{G_1}(w_1,w_2).$$
If a geodesic joining $w_1$ and $w_2$ has odd length, then $C$ defines a $v_0, v_0$-walk with odd length and
$$d_{G_1\times G_2}\big((w_1,v_0),(w_2,v_0) \big)\leq \max\{d_{G_1}(w_1,w_2),g_I(G_2)\}\le d_{G_1}(w_1,w_2)+g_I(G_2).$$

Thus, $i$ is a $\big(1,g_I(G_2)\big)$ quasi-isometric embedding.

Consider any $(w,v)\in V(G_1\times G_2)$. Then, if the geodesic joining $v$ and $v_0$ has even length,
$$d_{G_1\times G_2}\big((w,v),(w,v_0) \big)=d_{G_2}(v,v_0).$$
If a geodesic joining $v$ and $v_0$ has odd length, $[vv_0]\cup C$ defines a $v, v_0$-walk with even length. Therefore,
$$d_{G_1\times G_2}\big((w,v),(w,v_0) \big)\leq d_{G_2}(v,v_0)+g_I(G_2).$$
Thus, $i$ is $\big( \diam(V(G_2))+g_I(G_2)\big)$-full.

Hence, by Lemma \ref{l:vertices}, there is a $\big(\!\diam(V(G_2))+g_I(G_2)+\frac12 \big)$-full $\big(1,g_I(G_2)+1 \big)$-quasi-isometry, $j:G_1 \to G_1\times G_2$, and $G_1\times G_2$ is hyperbolic if and only if $G_1$ is hyperbolic by Theorem \ref{invarianza}.
\end{proof}

\begin{theorem}\label{t:no-odd}  Let $G_1$ be a graph without odd cycles and $G_2$ be a non-trivial bounded graph without odd cycles. Then, $G_1 \times G_2$ is hyperbolic if and only if $G_1$ is hyperbolic.
\end{theorem}

\begin{proof} Fix some vertex $w_0\in V(G_1)$ and some edge $[v_1,v_2]\in E(G_2)$.

By Theorem \ref{t:conect}, there are exactly two components in $G_1 \times G_2$. Since there are no odd cycles, there is no $(w_0,v_1), (w_0, v_2)$-walk in
$G_1\times G_2$. Thus, let us denote by $(G_1\times G_2)^1$ the component containing the vertex $(w_0,v_1)$ and  by $(G_1\times G_2)^2$ the component containing the vertex $(w_0,v_2)$.

Consider $i:V(G_1) \to V(G_1\times G_2)^1$ defined as $i(w):=(w,v_1)$ for every $w\in V(G_1)$ such that every $w_0, w$-walk has even length and  $i(w):=(w,v_2)$ for every $w\in V(G_1)$ such that every $w_0, w$-walk has odd length.

By Proposition \ref{p:distvert1}, $d_{G_1\times G_2}\big(i(w_1),i(w_2)\big)=d_{G_1}(w_1,w_2)$ for every $w_1,w_2\in V(G_1)$ and $i$ is a $(1,0)$-quasi-isometric embedding.

Let $(w,v)\in V(G_1\times G_2)^1$. Let $v_j$ with $j\in\{1,2\}$ such that every $v, v_j$-walk has even length. Then, by Proposition \ref{p:distvert1}, $d_{G_1\times G_2}\big((w,v),(w,v_j)\big)=d_{G_2}(v,v_j)\leq \diam(G_2)$. Therefore, $i$ is $\diam(G_2)$-full.

Hence, by Lemma \ref{l:vertices}, there is a  $\big(\!\diam(G_2)+\frac12 \big)$-full $\big(1,1 \big)$-quasi-isometry, $j:G_1 \to (G_1\times G_2)^1$, and $(G_1\times G_2)^1$ is hyperbolic if and only if $G_1$ is hyperbolic by Theorem \ref{invarianza}.

The same argument proves that $(G_1\times G_2)^2$ is hyperbolic.
\end{proof}

Denote by $P_2$ the path graph with two vertices, i.e., a graph with two vertices and an edge.

\begin{lemma} \label{l:P2}
Let $G_1$ be a graph with some odd cycle and $G_2$ a non-trivial bounded graph without odd cycles.
Then $G_1\times G_2$ and $G_1\times P_2$ are quasi-isometric and $\d(G_1\times P_2)\le \d(G_1\times G_2)$.
\end{lemma}

\begin{proof}
By Theorem \ref{t:conect}, we know that $G_1\times G_2$ and $G_1\times P_2$ are connected graphs.

Denote by $v_1$ and $v_2$ the vertices of $P_2$ and fix $[w_1,w_2] \in E(G_2)$.
The map $f: V(G_1\times P_2) \longrightarrow V(G_1\times [w_1,w_2])$ defined as $f(u,v_j):=(u,w_j)$ for every $u\in V(G_1)$ and $j=1,2,$
is an isomorphism of graphs; hence, it suffices to prove that
$G_1\times G_2$ and $G_1\times [w_1,w_2]$ are quasi-isometric.

Consider the inclusion map $i: V(G_1\times [w_1,w_2]) \longrightarrow V(G_1\times G_2)$.
Since $G_1\times [w_1,w_2]$ is a subgraph of $G_1\times G_2$, we have
$d_{G_1\times G_2}(x,y) \le d_{G_1\times [w_1,w_2]}(x,y)$ for every $x,y \in V(G_1\times [w_1,w_2])$.

Since $G_2$ is a graph without odd cycles, every $w_1, w_2$-walk has odd length and every $w_j, w_j$-walk has even length for $j=1,2$.
Thus Proposition \ref{p:distvert1} gives, for every $x=(u,w_1),y=(v,w_2) \in V(G_1\times [w_1,w_2])$,
$$
d_{G_1\times [w_1,w_2]}(x, y) = d_{G_1\times G_2}(x, y) = \min\big\{L(g)\,| \; g \text{ is a }   u, v \text{-walk of odd length} \big\}.
$$
Furthermore, for every $x=(u,w_j),y=(v,w_j) \in V(G_1\times [w_1,w_2])$ and $j=1,2,$
$$
d_{G_1\times [w_1,w_2]}(x, y) = d_{G_1\times G_2}(x, y) = \min\big\{L(g)\,| \; g \text{ is a }   u, v \text{-walk of even length} \big\}.
$$
Hence, $d_{G_1\times [w_1,w_2]}(x,y) = d_{G_1\times G_2}(x,y)$ for every $x,y \in V(G_1\times [w_1,w_2])$, and the inclusion map $i$ is an $(1,0)$-quasi-isometric embedding.
Therefore, $\d(G_1\times P_2)=\d(G_1\times [w_1,w_2])\le \d(G_1\times G_2)$.

Since $G_2$ is a graph without odd cycles, given any $w \in V(G_2)$, we have either that every $w, w_1$-walk has even length and every $w, w_2$-walk has odd length  or that every $w, w_2$-walk has even length and every $w, w_1$-walk has odd length. Also, since $G_1$ is connected, for each $u\in V(G_1)$ there is some $u'\in V(G_1)$ such that $[u,u']\in E(G_1)$. Therefore, by Proposition \ref{p:distvert1}, for every $(u,w) \in V(G_1\times G_2)$, if $\min \big\{ d_{G_2}(w,w_1), \, d_{G_2}(w,w_2) \big\}$ is even, then
$$d_{G_1\times G_2}\big((u,w),V(G_1\times [w_1,w_2])\big)=d_{G_1\times G_2}\big((u,w),V(u\times [w_1,w_2])\big)= \min \big\{ d_{G_2}(w,w_1), \, d_{G_2}(w,w_2) \big\},$$ and if $\min \big\{ d_{G_2}(w,w_1), \, d_{G_2}(w,w_2) \big\}$ is odd, then $$d_{G_1\times G_2}\big((u,w),V(G_1\times [w_1,w_2])\big)=d_{G_1\times G_2}\big((u,w),V(u'\times [w_1,w_2])\big)= \min \big\{ d_{G_2}(w,w_1), \, d_{G_2}(w,w_2) \big\}.$$

In both cases,
$$
\begin{aligned}
d_{G_1\times G_2}\big((u,w),V(G_1\times [w_1,w_2])\big)\le \diam V(G_2) ,
\end{aligned}
$$

\noindent and $i$ is $\big( \! \diam V(G_2)\big)$-full.
By Lemma \ref{l:vertices},
there exists a $\big( \!\diam V(G_2)+\frac{1}{2}\big)$-full $(1,1)$-quasi-isometry $g: G_1\times [w_1,w_2] \longrightarrow G_1\times G_2$.
\end{proof}

We say that a subgraph $\G$ of $G$ is \emph{isometric} if $d_{\G}(x,y)=d_{G}(x,y)$ for every $x,y\in \G$.
It is easy to check that a subgraph $\G$ of $G$ is isometric if and only if $d_{\G}(u,v)=d_{G}(u,v)$ for every $u,v\in V(\G)$.
Isometric subgraphs are very important in the study of hyperbolic graphs, as the following result shows.

\begin{lemma}\cite[Lemma 5]{RSVV}
\label{l:subgraph}
If $\G$ is an isometric subgraph of $G$, then $\d(\G) \le \d(G)$.
\end{lemma}

A $u, v$-walk $g$ in $G$ is a \emph{shortcut} of a cycle $C$ if $g \cap C=\{u,v\}$ and $L(g) < d_{C}(u,v)$ where $d_C$ denotes the length metric on $C$.

\smallskip

A cycle $C'$ is a \emph{reduction} of the cycle $C$ if both have odd length and $C'$ is the union of a subarc $\eta$ of $C$ and a shortcut of $C$ joining the endpoints of $\eta$.
Note that $L(C') \le L(C) - 2$.
We say that a cycle is \emph{minimal} if it has odd length and it does not have a reduction.

\begin{lemma}
\label{l:oddcycle}
If $C$ is a minimal cycle of $G$, then $L(C) \le 4\d(G)$.
\end{lemma}

\begin{proof}
We prove first that $C$ is an isometric subgraph of $G$.
Seeking for a contradiction assume that $C$ is not an isometric subgraph.
Thus, there exists a shortcut $g$ of $C$ with endpoints $u,v$.
There are two subarcs $\eta_1,\eta_2$ of $C$ joining $u$ and $v$; since $C$ has odd length, we can assume that $\eta_1$ has even length and $\eta_2$ has odd length.
If $g$ has even length, then $C':= g \cup \eta_2$ is a reduction of $C$.
If $g$ has odd length, then $C'':= g \cup \eta_1$ is a reduction of $C$.
Hence, $C$ is not minimal, which is a contradiction, and so $C$ is an isometric subgraph of $G$.

Let $x,y \in C$ with $d_C(x,y)=L(C)/2$ and $\s_1, \s_2$ the two subarcs of $C$ joining $x,y$.
Since $C$ is an isometric subgraph, $T:=\{\s_1, \s_2\}$ is a geodesic bigon.
If $p$ is the midpoint of $\s_1$, then Lemma \ref{l:subgraph} gives $L(C)/4 = d_{G}(p,\{x,y\})= d_{G}(p,\s_2) \le \d(C) \le \d(G)$.
\end{proof}

Given any $w_{0}, w_{k}$-walk $g = [w_{0},w_{1}] \cup [w_{1},w_{2}] \cup  \cdots  \cup [w_{k-1},w_{k}]$ in $G_1$ and $P_2=[v_1,v_2]$,
if $L(g)$ is either odd or even, then we define the $(w_0,v_1), (w_k, v_i)$-walk for $i\in {1,2}$,
$$
\begin{aligned}
\Gamma_1 g & := [(w_0,v_1),(w_1,v_2)] \cup  [(w_1,v_2),(w_2,v_1)] \cup [(w_2,v_1),(w_3,v_2)] \cup \cdots \cup [(w_{k-1},v_1),(w_k,v_2)],
\\
\Gamma_1 g & := [(w_0,v_1),(w_1,v_2)] \cup  [(w_1,v_2),(w_2,v_1)] \cup [(w_2,v_1),(w_3,v_2)] \cup \cdots \cup [(w_{k-1},v_2),(w_k,v_1)],
\end{aligned}
$$
respectively.

\begin{remark}\label{r:geodesic} By Proposition \ref{p:distvert1}, if $g$ is a geodesic path in $G_1$, then $\Gamma_1 g$ is a geodesic path in $G_1\times P_2$.
\end{remark}

Let us define the map $R: V(G_1\times P_2) \rightarrow V(G_1\times P_2)$ as
$R(w,v_1)=(w,v_2)$ and $R(w,v_2)=(w,v_1)$
for every $w \in V(G_1)$, and the path $\Gamma_2 g$ as $\Gamma_2 g=R(\Gamma_1 g)$.

Let us define the map $(\Gamma_1 g)': g \to \Gamma_1 g$  which is an isometry on the edges and such that $(\Gamma_1 g)'(w_{j})=(w_j,v_1)$ if $j$ is even and $(\Gamma_1 g)'(w_{j})=(w_j,v_2)$ if $j$ is odd. Also, let $(\Gamma_2 g)':g \rightarrow \Gamma_2 g$ be the map defined by $(\Gamma_2 g)':=R\circ (\Gamma_1 g)'$.

\smallskip

Given a graph $G$, denote by $\mathfrak{C}(G)$ the set of minimal cycles of $G$.

\begin{lemma} \label{l:dc}
Let $G_1$ be a graph with some odd cycle and $P_2=[v_1,v_2]$.
Consider a geodesic $g= [w_{0}w_{k}]= [w_{0},w_{1}] \cup [w_{1},w_{2}] \cup  \cdots  \cup [w_{k-1},w_{k}]$ in $G_1$.
Let us define $w_0':=(\Gamma_1 g)'(w_0)=(w_0,v_1)$ and $w_k':=(\Gamma_2 g)'(w_k)$, i.e., $w_k':=(w_k,v_1)$ or $w_k':=(w_k,v_2)$ if $k$ is odd or even, respectively.
Then $d_{G_1\times P_2}( w_0', w_k') > \sqrt{d_{G_1}\big( w_j, \mathfrak{C}(G_1) \big)}$ for every $0 \le j \le k$.
\end{lemma}

\begin{proof}
Fix $0 \le j \le k$.
Define
$$
\mathfrak{P}:= \big\{\s\,| \; \s \text{ is a } w_0, w_k \text{-walk } \text{ such that $L(\s)$ has a parity different from that of } k \big\} .
$$
Proposition \ref{p:distvert1} gives
$$
d_{G_1\times P_2}( w_0', w_k')
 = \min\big\{L(\s)\,| \; \s \in \mathfrak{P} \big\} .
$$
Choose $\s_0 \in \mathfrak{P}$ such that $L(\s_0) = d_{G_1\times P_2}( w_0', w_k')$.
Since $L(g)+L(\s_0)$ is odd, we have $L(g)+L(\s_0)=2t+1$ for some positive integer $t$.
Thus $d_{G_1\times P_2}(w_{0}',w_{k}')=L(\s_0)>\frac{1}{2}(2t+1)$.

If $g \cup \s_0$ is a cycle, then let us define $C_0:=g \cup \s_0$.
Thus, $L(C_0)=2t+1$ and $d_{G_1}\big( w_j,C_0 \big)=0$ for every $0 \le j \le k$.
Otherwise, we may assume that $g\cap \s_0=[w_0w_{i_1}]\cup [w_{i_2}w_k]$ for some $0\leq i_1<i_2\leq k$.
If $\s_1=\s_0\setminus g$, then let us define $C_0:=[w_{i_1}w_{i_2}]\cup \s_1$ (where $[w_{i_1}w_{i_2}]\subset g$).
Hence, $C_0$ is a cycle, $L(C_0)\leq 2t-1$ and $d_{G_1}\big( w_j,C_0 \big)  < \frac12 (2t+1) $.

If $C_0$ is not minimal, then consider
a reduction $C_1$ of $C_0$. Let us repeat the process until we obtain a minimal cycle $C_s$.
Note that $L(C_1)\le L(C_0) -2$ and for every point $p_1\in C_0$, $d_{G_1}\big( p_1, C_1 \big)  < \frac12 L(C_0) $.
Now, repeating the argument, for every $1< i \le s$, $L(C_i)\le L(C_{i-1}) -2$ and for every point $p_i \in C_{i-1}$,
$d_{G_1}\big( p_i, C_i \big) < \frac12 L(C_{i-1})$.
Therefore,
$$
\begin{aligned}
d_{G_1}\big( w_j, \mathfrak{C}(G_1) \big) \leq d_{G_1}\big( w_j, C_s \big)
&\le d_{G_1}\big( w_j, C_0 \big) + \frac12 L(C_0)  +  \frac12 L(C_1)  + \cdots + \frac12 L(C_s)
\\
& < \frac12 (2t+1)  +  \frac12 (2t-1)  + \cdots +\frac52 + \frac32\,.
\end{aligned}
$$

Hence, $$d_{G_1}\big( w_j, \mathfrak{C}(G_1) \big)<\frac12\sum_{i=1}^t (2i+1)=\frac12t^2+t <\Big(\frac12 (2t+1)\Big)^2 <
\Big( d_{G_1\times P_2}( w_0', w_k')\Big)^2.$$
\end{proof}

\begin{corollary} \label{c:dc}
Let $G_1$ be a hyperbolic graph with some odd cycle and $P_2=[v_1,v_2]$.
Consider a geodesic $g= [w_{0}w_{k}]= [w_{0},w_{1}] \cup [w_{1},w_{2}] \cup  \cdots  \cup [w_{k-1},w_{k}]$ in $G_1$.
Let us define $w_0':=(\Gamma_1 g)'(w_0)=(w_0,v_1)$ and $w_k':=(\Gamma_2 g)'(w_k)$.
Then, we have for every $0 \le j \le k$,
$$
\frac12 \Big( k + \sqrt{d_{G_1}\big( w_j, \mathfrak{C}(G_1) \big)} \, \Big)
\le d_{G_1\times P_2}( w_0', w_k')
\le k + 2d_{G_1}\big( w_j, \mathfrak{C}(G_1) \big) + 4\d(G_1).
$$
\end{corollary}

\begin{proof}
Corollary \ref{c:lower} and Lemma \ref{l:dc} give $d_{G_1\times P_2}( w_0', w_k') \ge k$ and $d_{G_1\times P_2}( w_0', w_k') \ge \sqrt{d_{G_1}\big( w_j, \mathfrak{C}(G_1) \big)} $,
and these inequalities provide the lower bound of $d_{G_1\times P_2}( w_0', w_k')$.

Consider a geodesic $\g$ joining $w_j$ and $C \in \mathfrak{C}(G_1)$ with $L(\g)=d_{G_1}( w_j, C )= d_{G_1}\big( w_j, \mathfrak{C}(G_1) \big)$
and the $w_0, w_k$-walk $$g':= [w_0w_j] \cup  \g \cup C \cup \g \cup [w_jw_k].$$
One can check that $\Gamma_1 g'$ is a $w_0', w_k'$-walk in $G_1\times P_2$, and so Lemma \ref{l:oddcycle} gives
$$d_{G_1\times P_2}( w_0', w_k')
\le L(\Gamma_1 g')
= L(g')
= k + 2d_{G_1}\big( w_j, \mathfrak{C}(G_1) \big) + L(C)
\le k + 2d_{G_1}\big( w_j, \mathfrak{C}(G_1) \big) + 4\d(G_1).$$
\end{proof}

If $[v_1,v_2]\in E(G)$, then we say that the point $x\in [v_1,v_2]$ with $d_{G}(x,v_1)=d_{G}(x,v_2)=1/2$ is the \emph{midpoint} of $[v_1,v_2]$.
Denote by $J(G)$ the set of vertices and midpoints of edges in $G$.
Consider the set $\mathbb{T}_1(G)$ of geodesic triangles $T$ in
$G$ that are cycles and
such that the three vertices of the triangle $T$ belong to $J(G)$, and
denote by $\delta_1(G)$ the infimum of the constants
$\lambda$ such that every triangle in $\mathbb{T}_1(G)$ is
$\lambda$-thin.

The following three results, which appear in \cite{BRS}, will be used throughout the paper.

\begin{theorem}\cite[Theorem 2.5]{BRS}
\label{t:main}
For every graph $G$ we have $\d_1(G)=\d(G)$.
\end{theorem}

The next result will narrow the posible values for the hyperbolicity constant $\delta$.

\begin{theorem}\cite[Theorem 2.6]{BRS}
\label{l:cuartos}
Let $G$ be any graph. Then $\delta(G)$ is always a multiple of $1/4$.
\end{theorem}

\begin{theorem}\cite[Theorem 2.7]{BRS}
\label{t:3}
For any hyperbolic graph $G$, there exists a geodesic triangle $T\in \mathbb{T}_1(G)$ such that $\delta(T)=\delta(G)$.
\end{theorem}

Consider the set $\mathbb{T}_v(G)$ of geodesic triangles $T$ in
$G$ that are cycles and
such that the three vertices of the triangle $T$ belong to $V(G)$, and
denote by $\delta_v(G)$ the infimum of the constants
$\lambda$ such that every triangle in $\mathbb{T}_v(G)$ is
$\lambda$-thin.

\begin{theorem} \label{t:main2}
For every graph $G$ we have $\d_v(G) \le \d(G) \le 4\d_v(G)+1/2$.
Hence, $G$ is hyperbolic if and only if $\d_v(G) < \infty$.
Furthermore, if $G$ is hyperbolic, then $\delta_v(G)$ is always a multiple of $1/2$ and there exist a geodesic triangle $T=\{x,y,z\}\in \mathbb{T}_v(G)$ and $p\in [xy] \cap J(G)$ such that
$d(p,[xz] \cup [zy])=\delta(T)=\delta_v(G)$.
\end{theorem}

\begin{proof}
The inequality $\d_v(G) \le \d(G)$ is direct.

Consider the set $\mathbb{T}_v'(G)$ of geodesic triangles $T$ in $G$
such that the three vertices of the triangle $T$ belong to $V(G)$, and
denote by $\delta_v'(G)$ the infimum of the constants
$\lambda$ such that every triangle in $\mathbb{T}_v'(G)$ is
$\lambda$-thin.
The argument in the proof of \cite[Lemma 2.1]{RT1} gives that $\delta_v'(G)=\delta_v(G)$.

In order to prove the upper bound of $\d(G)$, assume first that $G$ is hyperbolic.
We can assume $\delta_v'(G) < \infty$, since otherwise the inequality is direct.
By Theorem \ref{t:3}, there exists a geodesic triangle $T=\{x,y,z\}$ that is a cycle with $x,y,z \in J(G)$ and $p\in [xy]$ such that $d(p,[xz] \cup [zy])=\delta(T)=\delta(G)$.
Assume that $x,y,z \in J(G) \setminus V(G)$ (otherwise, the argument is simpler).
Let $x_1,x_2,y_1,y_2,z_1,z_2 \in T \cap V(G)$ such that $x \in [x_1,x_2], y \in [y_1,y_2], z \in [z_1,z_2]$ and $x_2,y_1 \in [xy], y_2,z_1 \in [yz], z_2,x_1 \in [xz]$.
Since $H:=\{x_2,y_1, y_2,z_1, z_2,x_1\}$ is a geodesic hexagon with vertices in $V(G)$, it is $4\d_v'(G)$-thin and every point $w\in [y_1,y_2] \cup [y_2z_1] \cup [z_1,z_2] \cup [z_2x_1] \cup [x_1,x_2]$ verifies
$d(w,[xz] \cup [zy]) \le 1/2$, we have
$$
\delta(G)=d(p,[xz] \cup [zy]) \le d(p,[y_1,y_2] \cup [y_2z_1] \cup [z_1,z_2] \cup [z_2x_1] \cup [x_1,x_2]) + 1/2 \le 4\d_v'(G) + 1/2 = 4\d_v(G) + 1/2.
$$
Assume now that $G$ is not hyperbolic.
Therefore, for each $M>0$ there exists a geodesic triangle $T=\{x,y,z\}$ that is a cycle with $x,y,z \in J(G)$ and $p\in [xy]$ such that $d(p,[xz] \cup [zy]) \ge M$.
The previous argument gives $M \le 4\d_v(G) + 1/2$ and, since $M$ is arbitrary, we deduce $\d_v(G) = \infty =\d(G)$.

\smallskip

Finally, consider any geodesic triangle $T=\{x,y,z\}$ in $\mathbb{T}_v(G)$.
Since $d(p,[xz] \cup [zy]) = d(p,([xz] \cup [zy])\cap V(G))$,
$d(p,[xz] \cup [zy])$ attains its maximum value when $p\in J(G)$.
Hence, $\delta(T)$ is a multiple of $1/2$ for every geodesic triangle $T \in \mathbb{T}_v(G)$.
Since the set of non-negative numbers that are multiple of $1/2$ is a discrete set,
if $G$ is hyperbolic, then $\delta(G)$ is a multiple of $1/2$ and there exist a geodesic triangle $T=\{x,y,z\}\in \mathbb{T}_v(G)$ and $p\in [xy] \cap J(G)$ such that
$d(p,[xz] \cup [zy])=\delta(T)=\delta_v(G)$.
This finishes the proof.
\end{proof}

\begin{theorem}\label{t:non-hyperbolic} If $G_1$ is a non-hyperbolic graph, then $G_1 \times P_2$ is not hyperbolic.
\end{theorem}

\begin{proof} Since $G_1$ is not hyperbolic, by Theorem \ref{t:main2}, given any $R>0$ there is a geodesic triangle $T=\{x,y,z\}$ that is a cycle, with $x,y,z\in V(G_1)$ and such that $T$ is not $R$-thin. Therefore, there exists some point $m\in T$, let us assume that $m\in [xy]$, such that $d_{G_1}(m,[yz]\cup [zx])>R$.

Seeking for a contradiction let us assume that $G_1\times P_2$ is $\delta$-hyperbolic.

Suppose that for some $R>\delta$, there is a geodesic triangle $T=\{x,y,z\}$ that is an even cycle in $G_1$, with $x,y,z\in V(G_1)$ and such that $T$ is not $R$-thin. Consider the (closed) path $\Lambda=[xy]\cup [yz]\cup [zx]$. Then, since $T$ has even length, the path $\Gamma_1\Lambda$ defines a cycle in $G_1\times P_2$. Let $\gamma_1$,
$\gamma_2$, $\gamma_3$ be the  paths in $\Gamma_1\Lambda$ corresponding to $[xy], [yz], [zx]$, respectively. By Corollary \ref{c:lower}, the curves $\gamma_1$, $\gamma_2$ and $\gamma_3$ are geodesics, and
$d_{G_1\times P_2}\big((\Gamma_1\Lambda)'(m),\gamma_2\cup \gamma_3\big)>\delta$, leading to contradiction.

Suppose that for every $R>0$, there is a geodesic triangle $T=\{x,y,z\}$ which is an odd cycle, with $x,y,z\in V(G_1)$ and such that $T$ is not $R$-thin.

Let $T_1=\{x,y,z\}$ be a geodesic triangle as above  and let us assume that $\diam(T_1)=D>8\delta$.

Let $T_2=\{x',y',z'\}$ be another geodesic triangle as above such that $T_2$ is not $3(D+8\delta)$-thin, this is, there is a point $m$ in one of the sides, let us call it $\sigma$, of $T_2$ such that $d_{G_1}(m,T_2\backslash \sigma)>3(D+8\delta)$.

Let $g=[w_0w_k]$ with $w_0\in T_1$ and $w_k\in T_2$ be a shortest geodesic in $G_1$ joining $T_1$ and $T_2$ (if $T_1$ and $T_2$ intersect, just assume that $g$ is a single vertex, $w_0=w_k$, in the intersection).

Let us assume that $w_0\in [xz]$ and $w_k\in[x'z']$. Then, let us consider the cycle $C$ in $G_1$ given by the union of the geodesics in $T_1$, $g$, the geodesics in $T_2$ and the inverse of $g$ from $w_k$ to $w_0$, this is,
$$C:=[w_0x]\cup [xy]\cup [yz]\cup[zw_0]\cup [w_0w_k] \cup [w_kx']\cup [x'y']\cup [y'z']\cup [z'w_k]\cup [w_kw_0].$$

Since $T_1,T_2$ are odd cycles,  $C$ is an even cycle. Therefore, $\Gamma_1C$ defines a cycle in $G_1\times P_2$. Moreover, by
Remark \ref{r:geodesic}, $\Gamma_1C$ is a geodesic decagon in $G_1\times P_2$ with sides $\gamma_1=(\Gamma_1C)'([w_0x])$, $\gamma_2=(\Gamma_1C)'([xy])$, $\gamma_3=(\Gamma_1C)'([yz])$, $\gamma_4=(\Gamma_1C)'([zw_0])$, $\gamma_5=(\Gamma_1C)'([w_0w_k])$, $\gamma_6=(\Gamma_1C)'([w_kx'])$, $\gamma_7=(\Gamma_1C)'([x'y'])$, $\gamma_8=(\Gamma_1C)'([y'z'])$, $\gamma_9=(\Gamma_1C)'([z'w_k])$ and $\gamma_{10}=(\Gamma_1C)'([w_kw_0])$.

Since we are assuming that $G_1\times P_2$ is $\delta$-hyperbolic, then for every $1\leq i \leq 10$ and every point $p \in \gamma_i$,  $d_{G_1\times P_2}(p,C\backslash \gamma_i)\le 8\delta$.

Let $p:=(\Gamma_1C)'(m)$.

Case 1. Suppose that $d_{G_1}(m,T_1\cup g)>8\delta$.

By assumption, $d_{G_1}(m,T_2\backslash \sigma)>8\delta$. If $\sigma= [x'y']$ (resp. $\sigma= [y'z']$), then $p\in \gamma_7$ (resp. $p\in \gamma_8$) and, by Corollary \ref{c:lower}, $d_{G_1\times P_2}(p,C\backslash \gamma_7)>8\delta$ (resp. $d_{G_1\times P_2}(p,C\backslash \gamma_8)>8\delta$) leading to contradiction. If $\sigma=[x'z']$, since $[x'z']=[x'w_k]\cup [w_kz']$, let us assume $m\in [x'w_k]$. Then,  since $d_{G_1}(m,w_k)>8\delta$, it follows that $d_{G_1}(m,[w_kz'])>8\delta$. Thus,  $p\in \gamma_6$ and, by Corollary \ref{c:lower}, $d_{G_1\times P_2}(p,C\backslash \gamma_6)>8\delta$ leading to contradiction.

Case 2. Suppose that $d_{G_1}(m,T_1\cup g)\le 8\delta$ and $L(g)\leq 8\delta$. Then, for every point $q$ in $T_1\cup g$, $d_{G_1}(m,q)\leq 8\delta +D +8\delta$. In particular, $d_{G_1}(m,w_k)\leq 8\delta +D +8\delta$. Therefore, $m\in [x'z']$ and let us assume that $m\in [x'w_k]$. Since $d_{G_1}(m,x')\ge d_{G_1}(m,[x'y']\cup[y'z'])>3(D+8\delta)$, there is a point $m'\in [x'm]\subset [x'w_k]$ such that $d_{G_1}(m,m')=2(D+8\delta)$. Then, $d_{G_1}(m',T_1\cup g)\ge 2(D+8\delta)-D-8\delta-8\delta=D>8\delta$. Also, it is trivial to check that $d_{G_1}(m',[x'y']\cup[y'z'])>3(D+8\delta)-2(D+8\delta)>8\delta$ and since $[x'z']$ is a geodesic, $d_{G_1}(m',[z'w_k])>8\delta$. Thus,  if $p':=(\Gamma_1C)'(m')$, then $p'\in \gamma_6$ and, by Corollary \ref{c:lower},  $d_{G_1\times P_2}(p',C\backslash \gamma_6)>8\delta$ leading to contradiction.

Case 3. Suppose that $d_{G_1}(m,T_1\cup g)\le 8\delta$ and $L(g)> 8\delta$. Since $g$ is a shortest geodesic in $G_1$ joining $T_1$ and $T_2$, this implies that $d_{G_1}(T_1,T_2)>8\delta$ and $d_{G_1}(m,[w_0w_k])\leq 8\delta$. Moreover, $d_{G_1}(m,w_k)\leq 16\delta$. Otherwise, there is a point $q\in [w_0w_k]$ such that $d_{G_1}(m,q)\le 8\delta$ and $d_{G_1}(q,w_k)>8 \delta$ which means that $d_{G_1}(q,w_0)< d_{G_1}(w_0,w_k) - 8 \delta$ and $d_{G_1}(m,w_0)< d_{G_1}(w_0,w_k)$ leading to contradiction.

Since $d_{G_1}(m,w_k)\leq 16\delta$, $m\in [x'z']$. Let us assume that $m\in [x'w_k]$. Since $d_{G_1}(m,[x'y']\cup[y'z'])>3(D+8\delta)$, there is a point $m'\in [x'm]\subset [x'w_k]$ such that $d_{G_1}(m,m')=2(D+8\delta)$. Let us see that $d_{G_1}(m', [w_0w_k])>8\delta$. Suppose there is some $q\in [w_0w_k]$ such that $d_{G_1}(m',q)\leq 8\delta$. Since $m'\in T_2$ and $g$ is a shortest geodesic joining $T_1$ and $T_2$, $d_{G_1}(q,w_k)\leq 8\delta$. However,
$32\d < 2(D+8\delta)=d_{G_1}(m',m)\leq d_{G_1}(m',q)+d_{G_1}(q,w_k)+d_{G_1}(w_k,m)\le 8\delta + 8\delta + 16\delta$ which is a contradiction. Hence, $d_{G_1}(m', [w_0w_k])>8\delta$. Also, it is trivial to check that $d_{G_1}(m',[x'y']\cup[y'z'])>3(D+8\delta)-2(D+8\delta)>8\delta$ and since
$[x'z']$ is a geodesic, $d_{G_1}(m',[z'w_k])>8\delta$. Thus,  if $p':=(\Gamma_1C)'(m')$, then $p'\in \gamma_6$ and, by Corollary \ref{c:lower},
$d_{G_1\times P_2}(p',C\backslash \gamma_6)>8\delta$ leading to contradiction.
\end{proof}

Proposition \ref{p:unbounded}, Lemma \ref{l:P2} and Theorems \ref{t:G2odd}, \ref{t:no-odd} and \ref{t:non-hyperbolic} have the following consequence.

\begin{corollary}\label{c:non-hyperbolic} If $G_1$ is a non-hyperbolic graph and $G_2$ is some non-trivial graph, then $G_1 \times G_2$ is not hyperbolic.
\end{corollary}

Proposition \ref{p:unbounded} and Corollary \ref{c:non-hyperbolic} provide a necessary condition for the hyperbolicity of $G_1\times G_2$.

\begin{theorem}\label{t:neccon}
Let $G_1, G_2$ be non-trivial graphs. If $G_1\times G_2$ is hyperbolic, then one factor graph is hyperbolic and the other one is bounded.
\end{theorem}

Theorems \ref{t:G2odd} and \ref{t:no-odd} show that this necessary condition is also sufficient if either $G_2$ has some odd cycle or $G_1$ and $G_2$ do not have odd cycles
(when $G_1$ is a hyperbolic graph and $G_2$ is a bounded graph).
We deal now with the other case, when $G_1$ has some odd cycle and $G_2$ does not have odd cycles.

\begin{theorem} \label{t:P2}
Let $G_1$ be a graph with some odd cycle and $G_2$ a non-trivial bounded graph without odd cycles.
Assume that $G_1$ satisfies the following property:
for each $M>0$ there exist a geodesic $g$ joining two minimal cycles of $G_1$ and a vertex $u\in g \cap V(G_1)$ with
$d_{G_1}\big( u, \mathfrak{C}(G_1) \big) \ge M$.
Then $G_1\times G_2$ is not hyperbolic.
\end{theorem}

\begin{proof}
If $G_1$ is not hyperbolic, then Corollary \ref{c:non-hyperbolic} gives that $G_1\times G_2$ is not hyperbolic.
Assume now that $G_1$ is hyperbolic.
By Theorem \ref{invarianza} and Lemma \ref{l:P2}, we can assume that $G_2 = P_2$ and $V(P_2)= \{v_1,v_2\}$.

Fix $M>0$ and choose a geodesic $g= [w_{0}w_{k}]= [w_{0},w_{1}] \cup [w_{1},w_{2}] \cup  \cdots  \cup [w_{k-1},w_{k}]$ joining two minimal cycles in $G_1$ and $0 < r < k$ with
$d_{G_1}\big( w_r, \mathfrak{C}(G_1) \big) \ge M$.

%Denote by $v_1$ and $v_2$ the vertices of $P_2$.
Define the paths $g_1$ and $g_2$ in $G_1\times P_2$ as $g_1:= \Gamma_1 g$ and $g_2:= \Gamma_2 g$.
Since $L(g_1)=L(g_2)=L(g) = d_{G_1}(w_0,w_k)$, we have
$$
d_{G_1\times P_2}\big(g_1(w_0),g_1(w_k)\big) \le L(g_1) = d_{G_1}(w_0,w_k) ,
\qquad
d_{G_1\times P_2}\big(g_2(w_0),g_2(w_k)\big) \le L(g_2) = d_{G_1}(w_0,w_k).
$$
Corollary \ref{c:lower} gives that
$$
d_{G_1\times P_2}\big(g_1(w_0),g_1(w_k)\big) \ge d_{G_1}(w_0,w_k),
\qquad
d_{G_1\times P_2}\big(g_2(w_0),g_2(w_k)\big) \ge d_{G_1}(w_0,w_k).
$$
Hence, $g_1$ and $g_2$ are geodesics in $G_1\times P_2$.
Choose geodesics $g_3=[g_1(w_0) g_2(w_0)]$ and $g_4=[g_1(w_k) g_2(w_k)]$ in $G_1\times P_2$.
Since $d_{P_2}(v_1,v_2)=1$ is odd, Proposition \ref{p:distvert1} gives
$$
\begin{aligned}
d_{G_1\times P_2}\big(g_1(w_0),g_2(w_0)\big)
& = \min\big\{L(\s)\,| \; \s \text{ is a } w_0, w_0 \text{-walk }\big\}
\\
& = \min\big\{L(\s)\,| \; \s \text{ cycle of odd length containing } w_0 \big\}.
\end{aligned}
$$
Since $w_0$ belongs to a minimal cycle,
$L(g_3) \le 4\d(G_1)$ by Lemma \ref{l:oddcycle}.
In a similar way, we obtain $L(g_4) \le 4\d(G_1)$.

Consider the geodesic quadrilateral $Q := \{g_1, g_2, g_3, g_4\}$ in $G_1\times P_2$.
Thus
$d_{G_1\times P_2}\big(g_1(w_r),g_2 \cup g_3 \cup g_4\big) \le 2\d(G_1\times P_2)$.
Since $\max \big\{L(g_3), L(g_4)\big\} \le 4\d(G_1)$, we deduce
$d_{G_1\times P_2}\big(g_1(w_r),g_2 \big) \le 2\d(G_1\times P_2) + 4\d(G_1)$.

Let $0 \le j \le k$ with $d_{G_1\times P_2}\big(g_1(w_r),g_2 \big) = d_{G_1\times P_2}\big(g_1(w_r),g_2(w_j)\big)$.
Let us define $w_r':=g_1(w_r)$ and $w_j':=g_2(w_j)$.
Thus Lemma \ref{l:dc} gives
$$
\sqrt{M}
\le \sqrt{d_{G_1}\big( w_r, \mathfrak{C}(G_1) \big)}
\le d_{G_1\times P_2}( w_r', w_j')
= d_{G_1\times P_2}( w_r', g_2 )
\le 2\d(G_1\times P_2) + 4\d(G_1),
$$
and since $M$ is arbitrarily large, we deduce that $G_1\times P_2$ is not hyperbolic.
\end{proof}

\begin{lemma}\label{l:dense} Let $G_1$ be a hyperbolic graph and suppose there is some constant $K>0$ such that for every vertex $w\in G_1$, $d_{G_1}(w, \mathfrak{C}(G_1))\le K$. Then, $G_1\times P_2$ is hyperbolic.
\end{lemma}

\begin{proof} Denote by $v_1$ and $v_2$ the vertices of $P_2$. Let $i: V(G_1)\to V(G_1\times P_2)$ defined as $i(w):=(w,v_1)$ for every $w\in G_1$.

For every pair of vertices $x,y\in V(G_1)$, by Corollary \ref{c:lower}, $d_{G_1}(x,y)\leq d_{G_1\times P_2}(i(x),i(y))$.
By Corollary \ref{c:dc},
$$d_{G_1\times P_2}(i(x),i(y))\leq d_{G_1}(x,y)+2d_{G_1}\big(x,\mathfrak{C}(G_1)\big)+4\delta(G_1)\leq d_{G_1}(x,y)+2K+4\delta(G_1).$$
Therefore, $i: V(G_1)\to V(G_1\times P_2)$ is a $\big(1,2K+4\delta(G_1)\big)$-quasi-isometric embedding.

Notice that for every  $(w,v_1)\in V(G_1\times P_2)$, $(w,v_1)=i(w)$. Also, for any $(w,v_2)\in V(G_1\times P_2)$, since $G_1$ is connected, there is some edge $[w,w']\in E(G_1)$ and we have  $[(w,v_2),(w',v_1)]\in E(G_1\times P_2)$. Therefore, $i:V(G_1) \to V(G_1\times P_2)$ is 1-full.

Thus, by Lemma \ref{l:vertices}, $G_1$ and $G_1\times P_2$ are quasi-isometric and, by Theorem \ref{invarianza}, $G_1\times P_2$ is hyperbolic.
\end{proof}

Theorem \ref{t:G2odd} and Lemmas \ref{l:P2} and \ref{l:dense} have the following consequence.

\begin{theorem}\label{t:dense} Let $G_1$ be a hyperbolic graph and $G_2$ some non-trivial bounded graph.
If there is some constant $K>0$ such that for every vertex $w\in G_1$, $d_{G_1}(w, \mathfrak{C}(G_1))\le K$, then $G_1\times G_2$ is hyperbolic.
\end{theorem}

We will finish this section with a characterization of the hyperbolicity of $G_1\times G_2$, under an additional hypothesis.
Since the proof of this result is long and technical, in order to make the arguments more transparent, we
collect some results we need along the proof in technical lemmas.

Let $J$ be a finite or infinite index set.
Now, given a graph $G_1$, we define some graphs related to $G_1$ which will be useful in the following results.
Let $B_j:=B_{G_1}(w_j,K_j)$ with $w_j\in V(G_1)$ and $K_j\in \mathbb{Z}^+$, for any $j\in J$, such that $\sup_j K_j=K<\infty, \, \overline B_{j_1}\cap \overline B_{j_2}=\emptyset$ if $j_1\neq j_2$, and every odd cycle $C$ in $G_1$ satisfies $C \cap B_j\neq \emptyset$ for some $j\in J$.
Denote by $G_1'$ the subgraph of $G_1$ induced by $V(G_1)\setminus(\cup_j B_{j})$.
Let $N_j:=\p B_j=\{w\in V(G_1): d_{G_1}(w,w_j)=K_j\}$.
Denote by $G_1^{*}$ the graph with $V(G_1^{*})=V(G_1')\cup (\cup_j \{w_j^{*}\})$, where $w_j^{*}$ are additional vertices, and $E(G_1^{*})=E(G_1')\cup (\cup_j \{[w,w_j^{*}]: w\in N_j\})$. We have $G_1'=G_1\cap G_1^{*}$.

\begin{lemma}\label{l:1}
Let $G_1$ be a graph as above. Then, there exists a quasi-isometry $g: G_1\to G_1^{*}$ with $g(w_j)=w_j^{*}$ for every $j\in J$.
\end{lemma}

\begin{proof}
Let $f: V(G_1)\to V(G_1^{*})$ defined as $f(u)=u$ for every $u\in V(G_1')$, and $f(u)=w_i^{*}$ for every $u\in V(B_i)$.
It is clear  that $f: V(G_1)\to V(G_1^{*})$  is 0-full.

Now, we focus on proving that $f: V(G_1)\to V(G_1^{*})$ is a $(K,2K)$-quasi-isometric embedding.
For every $u,v\in V(G_1)$, it is clear that $d_{G_1^{*}}(f(u),f(v))\le d_{G_1}(u,v)$.

In order to prove the other inequality, let us fix $u,v\in V(G_1)$ and let us consider a geodesic $\g$ in $G_1^{*}$ joining $f(u)$ and $f(v)$.

Assume that $u,v\in V(G_1')$.
If $L(\g)= d_{G_1}(u,v)$, then $d_{G_1}(u,v)=d_{G_1^{*}}(f(u),f(v))$. If $L(\g)< d_{G_1}(u,v)$, then $\g$ meets some $w_j^{*}$.
Since $\g$ is a compact set, it intersects just a finite number of $w_j^{*}$'s, which we denote by $w_{j_1}^{*}, \ldots w_{j_r}^{*}$.
We consider $\g$ as an oriented curve from $f(u)$ to $f(v)$; thus we can assume that $\g$ meets $w_{j_1}^{*}, \ldots w_{j_r}^{*}$ in this order.

Let us define the following vertices in $\g$
$$w_i^1=[f(u)w_{j_i}^{*}]\cap N_{j_i},\quad \quad w_i^2=[w_{j_i}^{*}f(v)]\cap N_{j_i},$$
for every $1\le i\le r$.
Note that $[w_i^2w_{i+1}^1]\subset G_1'$ for every $1\le i< r$ (it is possible to have $w_i^2=w_{i+1}^1$).

Since $d_{G_1^{*}}(w_i^1,w_i^2)=2$ and $d_{G_1}(w_i^1,w_i^2)\le 2K$, we have $d_{G_1^{*}}(w_i^1,w_i^2) \ge \frac{1}{K} \, d_{G_1}(w_i^1,w_i^2)$ for every $1\le i\le r$.
Thus,
$$
\begin{aligned}
d_{G_1^{*}}(f(u),f(v)) & = d_{G_1^{*}}(f(u),w_1^1) + \sum_{i=1}^r d_{G_1^{*}}(w_i^1,w_i^2) + \sum_{i=1}^{r-1} d_{G_1^{*}}(w_i^2,w_{i+1}^1) + d_{G_1^{*}}(w_r^2, f(v))
\\
& \ge d_{G_1}(u,w_1^1) + \frac{1}{K} \sum_{i=1}^r d_{G_1}(w_i^1,w_i^2) + \sum_{i=1}^{r-1} d_{G_1}(w_i^2,w_{i+1}^1) + d_{G_1}(w_r^2, v)
\\
& \ge \frac{1}{K}\, \Big(d_{G_1}(u,w_1^1) +\sum_{i=1}^r d_{G_1}(w_i^1,w_i^2) + \sum_{i=1}^{r-1} d_{G_1}(w_i^2,w_{i+1}^1)+d_{G_1}(w_r^2, v)\Big)
\\
& \ge \frac{1}{K} \,d_{G_1}(u,v).
\end{aligned}
$$

Assume that $f(u)=f(v)$.
Therefore, there exists $j$ with $u,v\in B_j$ and
$$d_{G_1^{*}}(f(u),f(v))=0 > d_{G_1}(u,v)-2K.$$

Assume now that $u$ and/or $v$ does not belong to $V(G_1')$ and $f(u)\neq f(v)$.
Let $u_0,v_0$ be the closest vertices in $V(G_1')\cap \g$ to $f(u), f(v)$, respectively (it is possible to have $u_0=f(u)$ or $v_0=f(v)$).
Since $u_0,v_0\in V(G_1')$, $u_0=f(u_0), v_0=f(v_0)$, we have $d_{G_1}(u,u_0)< 2K$ and $d_{G_1}(v,v_0)< 2K$. Hence,
$$
\begin{aligned}
d_{G_1^{*}}(f(u),f(v))
& = d_{G_1^{*}}(f(u),u_0)+d_{G_1^{*}}(u_0,v_0)+d_{G_1^{*}}(v_0,f(v))
\\
& \ge d_{G_1^{*}}(f(u_0),f(v_0))
\\
& \ge \frac{1}{K}\, d_{G_1}(u_0,v_0)
\\
& \ge \frac{1}{K}\, \Big(d_{G_1}(u,v)- d_{G_1}(u,u_0)-d_{G_1}(v,v_0) \Big)
\\
& > \frac{1}{K}\, d_{G_1}(u,v)-4.
\end{aligned}
$$

If $K\ge 2$, then $d_{G_1^{*}}(f(u),f(v))>\frac{1}{K}\, d_{G_1}(u,v)-2K$.
If $K=1$, then $d_{G_1}(u,u_0)\le 1, d_{G_1}(v,v_0)\le 1$, and $d_{G_1^{*}}(f(u),f(v))\ge d_{G_1}(u,v)- 2$.

Finally, we conclude that $f: V(G_1)\to V(G_1^{*})$ is a $(K,2K)$-quasi-isometric embedding. Thus, Lemma \ref{l:vertices} provides a quasi-isometry $g: G_1\to G_1^{*}$ with the required property.
\end{proof}

\begin{definition} Given a graph $G_1$ and some index set $J$ let $\mathcal{B}_J=\{B_j\}_{j\in J}$
be a family of balls where $B_j:=B_{G_1}(w_j,K_j)$ with $w_j\in V(G_1)$, $K_j\in \mathbb{Z}^+$ for
any $j\in J$, $\sup_j K_j=K<\infty$ and $\overline B_{j_1}\cap \overline B_{j_2}=\emptyset$
if $j_1\neq j_2$. Suppose that every odd cycle $C$ in $G_1$ satisfies that $C \cap B_j\neq \emptyset$
for some $j\in J$.
If there is some constant $M>0$ such that for every $j\in J$, there is an odd cycle $C_j$ such that $C_j \cap B_j \neq \emptyset$ with $L(C_j) <M$, then
we say that $\mathcal{B}_J$ is $M$-\emph{regular}.
\end{definition}

\begin{remark}\label{r:M-regular}
If $J$ is finite, then there exists $M>0$ such that $\{B_j\}_{j\in J}$ is $M$-regular.
\end{remark}

Denote by $G^{*}$ the graph with $V(G^{*})=V(G_1'\times P_2)\cup (\cup_j \{w_j^{*}\})$, where $G_1'$ is a graph as above and $w_j^{*}$ are additional vertices, and $E(G^{*})= E(G_1'\times P_2)\cup (\cup_j \{[w,w_j^{*}]: \pi_1(w)\in N_j\})$.

\begin{lemma}\label{l:2}
Let $G_1$ be a graph as above and $P_2$ with $V(P_2)=\{v_1,v_2\}$.
If $G_1$ is hyperbolic and $\mathcal{B}_J$ as above is $M$-regular,
then there exists a quasi-isometry $f: G_1\times P_2\to G^{*}$ with $f(w_j,v_i)=w_j^{*}$ for every $j\in J$ and $i\in \{1,2\}$.
\end{lemma}

\begin{proof}
Let $F: V(G_1\times P_2)\to V(G^{*})$ defined as $F(v,v_i)=(v,v_i)$ for every $v\in V(G_1')$, and $F(v,v_i)=w_j^{*}$ for every $v\in V(B_j)$.
It is clear that $F: V(G_1\times P_2)\to V(G^{*})$ is 0-full.
Recall that we denote by $\pi_1: G_1\times P_2 \to G_1$ the projection map.
Define $\pi^{*}: G^{*}\to G_1$ as $\pi^{*}=\pi_1$ on $G_1'\times P_2$ and
$\pi^{*}(x)=w_j$ for every $x$ with $d_{G^*}(x,w_j^*)<1$ for some $j\in J$.

Now, we focus on proving that $F: V(G_1\times P_2)\to V(G^{*})$ is a quasi-isometric embedding.
For every $(w,v_i), (w',v_{i'})\in V(G_1\times P_2)$, one can check
$$d_{G^{*}}(F(w,v_i), F(w',v_{i'}))\le d_{G_1\times P_2}((w,v_i), (w',v_{i'})).$$
In order to prove the other inequality, let us fix $(w,v_i), (w',v_{i'})\in V(G_1'\times P_2)$ (the inequalities in other cases can be obtained from the one in this case, as in the proof of Lemma \ref{l:1}).
Consider a geodesic $\g:=[F(w,v_i) F(w',v_{i'})]$ in $G^{*}$.
If $L(\g)= d_{G_1\times P_2}((w,v_i),(w',v_{i'}))$, then
$$d_{G^{*}}(F(w,v_i), F(w',v_{i'}))= d_{G_1\times P_2}((w,v_i), (w',v_{i'})).$$
If $L(\g) < d_{G_1\times P_2}((w,v_i),(w',v_{i'}))$, then $\pi^{*}(\g)$ meets some $B_j$.
Since $\g$ is a compact set, $\pi^{*}(\g)$ intersects just a finite number of $B_j$'s, which we denote by $B_{j_1}, \ldots B_{j_r}$.
We consider $\g$ as an oriented curve from $F(w,v_i)$ to $F(w',v_{i'})$; thus we can assume that $\pi^{*}(\g)$ meets $B_{j_1}, \ldots B_{j_r}$ in this order.

Let us define the following set of vertices in $\g$
$$\{w_i^1, w_i^2\}:= \g\cap (N_{j_i}\times P_2),$$
for every $1\le i\le r$, such that $d_{G_1\times P_2}((w,v_i),w_i^1)< d_{G_1\times P_2}((w,v_i),w_i^2)$.
Note that $[w_i^2w_{i+1}^1]\subset G_1'\times P_2$ for every $1\le i< r$ and $d_{G_1\times P_2}(w_i^2,w_{i+1}^1)\ge 1$ since $\overline B_{j_i}\cap \overline B_{j_{i+1}}=\emptyset$.

If $d_{G_1}(\pi(w_i^1),\pi(w_i^2))=d_{G_1\times P_2}(w_i^1,w_i^2)$  for some $1\le i\le r$, then $d_{G_1\times P_2}(w_i^1,w_i^2)\le 2K$.
Since \\
$d_{G_1\times P_2}(w_i^2,w_{i+1}^1)\ge 1$  for $1\le i< r$,
we have that $d_{G_1\times P_2}(w_i^1,w_i^2)\le 2K\, d_{G_1\times P_2}(w_i^2,w_{i+1}^1)$ in this case.

If $d_{G_1}(\pi_1(w_i^1),\pi_1(w_i^2)) < d_{G_1\times P_2}(w_i^1,w_i^2)$ for some $1\le i\le r$, then $d_{G_1}(\pi_1(w_i^1),\pi_1(w_i^2)) + d_{G_1\times P_2}(w_i^1,w_i^2)$ is odd.

Since $\mathcal{B}_J$ is $M$-regular, consider an odd cycle $C$ with $C\cap B_{j_i}\neq \emptyset$ and $L(C)<M$, and let  	 			
$b_i\in C\cap B_{j_i}$ and $[\pi_1(w_i^1)b_i], [b_i\pi_1(w_i^2)]$ geodesics in $G_1$.				
Thus, $[\pi_1(w_i^1)b_i]\cup [b_i\pi_1(w_i^2)]$ and $[\pi_1(w_i^1)b_i]\cup C \cup [b_i\pi_1(w_i^2)]$
have different parity which means that one of them has different parity from $[\pi_1(w_i^1)\pi_1(w_i^2)]$.
Then, $d_{G_1\times P_2}(w_i^1,w_i^2)\le L([\pi_1(w_i^1)b_i]\cup C \cup [b_i\pi_1(w_i^2)]) \le 4K+ M$.
Since $d_{G_1\times P_2}(w_i^2,w_{i+1}^1)\ge 1$ for $1\le i< r$, we have that
$d_{G_1\times P_2}(w_i^1,w_i^2)\le \Big(4K+M\Big)\,d_{G_1\times P_2}(w_i^2,w_{i+1}^1)$ in this case.

Thus, we have that $d_{G_1\times P_2}(w_i^1,w_i^2)\le 4K+M$ for every $1\le i\le r$
and $d_{G_1\times P_2}(w_i^1,w_i^2)\le \Big(4K+M\Big)\,d_{G_1\times P_2}(w_i^2,w_{i+1}^1)$ for every $1\le i< r$.
Therefore,
\begin{eqnarray*}
\hspace{-0.5cm}&&d_{G_1\times P_2}((w,v_i),(w',v_{i'}))\le d_{G_1\times P_2}((w,v_i), w_1^1)+\sum_{i=1}^r d_{G_1\times P_2}(w_i^1, w_i^2)+\sum_{i=1}^{r-1} d_{G_1\times P_2}(w_i^2, w_{i+1}^1)\\
\hspace{-0.5cm}&&\hspace{4cm}+ \,d_{G_1\times P_2}(w_r^2,(w',v_{i'}))\\
\hspace{-0.5cm}&&\le  d_{G_1\times P_2}((w,v_i),w_1^1)+d_{G_1\times P_2}(w_r^2,(w',v_{i'}))+\Big(4K+M+1\Big)\sum_{i=1}^{r-1} d_{G_1\times P_2}(w_i^2,w_{i+1}^1)\\
\hspace{-0.5cm}&& \hspace{4cm} + \, d_{G_1\times P_2}(w_r^1,w_r^2)\\
\hspace{-0.5cm}&&=  d_{G^{*}}(F(w,v_i),F(w_1^1))+d_{G^{*}}(F(w_r^2),F(w',v_{i'}))+\Big(4K+M+1\Big)\sum_{i=1}^{r-1} d_{G^{*}}(F(w_i^2),F(w_{i+1}^1)) + d_{G_1\times P_2}(w_r^1,w_r^2)\\
\hspace{-0.5cm}&&\le \Big(4K+M+1\Big)\Big( d_{G^{*}}(F(w,v_i),F(w_1^1))+d_{G^{*}}(F(w_r^2),F(w',v_{i'}))+\sum_{i=1}^{r-1} d_{G^{*}}(F(w_i^2),F(w_{i+1}^1))\Big) + 4K+M\\
\hspace{-0.5cm}&&\le \Big(4K+M+1\Big)\, d_{G^{*}}(F(w,v_i),F(w',v_{i'}))+4K+M.
\end{eqnarray*}
We conclude that $F: V(G_1\times P_2)\to V(G^{*})$ is a quasi-isometric embedding.
Thus, Lemma \ref{l:vertices} provides a quasi-isometry $f: G_1\times P_2\to G^{*}$ with the required property.
\end{proof}

\begin{definition}
\label{def:3.1}
Given a geodesic metric space $X$ and closed connected pairwise disjoint
subsets $\{\eta_j\}_{j\in J}$ of $X$, we consider another copy
$X'$ of $X$. The \emph{double} $DX$ of $X$ is the union of $X$ and
$X'$ obtained by identifying the corresponding points in each
$\eta_j$ and $\eta_j'$.
\end{definition}

\begin{definition} \label{def:Hausdorff}
Let us consider $H>0$, a metric space $X$, and subsets
$Y,Z\subseteq X$. The set $V_H(Y):=\{x\in X:\,d(x,Y)\le H\}$ is
called the $H$-\emph{neighborhood} of $Y$ in $X$. The
\emph{Hausdorff distance} of $Y$ to $Z$ is defined by
$\mathcal{H}(Y,Z):=\inf\{H>0:\,Y\subseteq V_H(Z),\,\,Z\subseteq
V_H(Y)\}$.
\end{definition}

The following results in \cite{APRT}  and \cite{GH} will be useful.

\begin{theorem}\cite[Theorem 3.2]{APRT}
\label{th:3.1}
Let us consider a geodesic metric space $X$ and closed connected pairwise
disjoint subsets $\{\eta_j\}_{j\in J}$ of $X$, such that the
double $DX$ is a geodesic metric space. Then the following
conditions are equivalent:

\begin{enumerate}
\item[$(1)$] $DX$ is hyperbolic.

\item[$(2)$] $X$ is hyperbolic and there exists a constant $c_1$
such that for every $k,l \in J$ and $a\in \eta_k, b\in \eta_l$ we
have $d_X(x,\cup_{j\in J}\eta_j) \le c_1$ for every $x\in
[ab]\subset X$.

\item[$(3)$] $X$ is hyperbolic and there exist constants
$c_2,\a,\b$ such that for every $k,l \in J$ and $a\in \eta_k,
b\in \eta_l$ we have $d_X(x,\cup_{j\in J}\eta_j) \le c_2$ for
every $x$ in some $(\a,\b)$-quasi-geodesic joining $a$ with $b$ in
$X$.
\end{enumerate}
\end{theorem}

\begin{theorem}
\label{t:Geodesic-Stability}
\cite[p.87]{GH}
For each $\d\ge 0$, $a\ge 1$ and $b\ge 0$,
there exists a constant $H=H(\d,a,b)$ with the following property:

Let us consider a $\d$-hyperbolic geodesic metric space $X$ and
an $(a,b)$-quasigeodesic $g$ starting in $x$ and finishing in $y$.
If $\g$ is a geodesic joining $x$ and $y$, then
$\mathcal{H}(g,\g)\le H$.
\end{theorem}

This property is known as geodesic stability.
Mario Bonk proved in 1996 that geodesic stability was, in fact, equivalent to Gromov hyperbolicity (see \cite{Bo}).

\begin{theorem} \label{t:char}
Let $G_1$ be a graph and $B_j:=B_{G_1}(w_j,K_j)$ with $w_j\in V(G_1)$ and $K_j\in \mathbb{Z}^+$, for any $j\in J$, such that $\sup_j K_j=K<\infty$, $\overline B_{j_1}\cap \overline B_{j_2}=\emptyset$ if $j_1\neq j_2$, and every odd cycle $C$ in $G_1$ satisfies $C \cap B_j\neq \emptyset$ for some $j\in J$.
Suppose $\{B_j\}_{j\in J}$ is $M$-regular for some $M>0$.
Let $G_2$ be a non-trivial bounded graph without odd cycles. Then, the following statements are equivalent:

\begin{enumerate}
\item[$(1)$] $G_1\times G_2$ is hyperbolic.

\item[$(2)$]  $G_1$ is hyperbolic and there exists a constant $c_1$, such that for every $k,l\in J$ and $w_k\in B_k$, $w_l\in B_l$ there exists a geodesic $[w_kw_l]$ in $G_1$ with $d_{G_1}(x, \cup_{j\in J}w_j)\le c_1$ for every $x\in [w_kw_l]$.

\item[$(3)$]  $G_1$ is hyperbolic and there exist constants $c_2, \a, \b$, such that for every $k,l\in J$ we have $d_{G_1}(x, \cup_{j\in J}w_j)\le c_2$ for every $x$ in some ($\a,\b$)-quasi-geodesic joining $w_k$ with $w_l$ in $G_1$.
\end{enumerate}
\end{theorem}

\begin{proof}
Items $(2)$ and $(3)$ are equivalent by geodesic stability in $G_1$ (see Theorem \ref{t:Geodesic-Stability}).

Assume that $(2)$ holds. By Lemma \ref{l:1}, there exists an $(\a, \b)$-quasi-isometry $f: G_1 \to G_1^{*}$ with $f(w_j)=w_j^{*}$ for every $j\in J$.
Given $k,l\in J, f([w_kw_l])$ is an $(\a,\b)$-quasi-geodesic with endpoints $w_k^{*}$ and $w_l^{*}$ in $G_1^{*}$.
Given $x\in f([w_kw_l])$, we have $x=f(x_0)$ with $x_0\in [w_kw_l]$ and $d_{G_1^{*}}(x,\cup_{j\in J} w_j^{*})\le \a d_{G_1}(x_0, \cup_{j\in J} w_j) + \b\le \a c_1+\b$.
Taking $X=G_1^{*}, DX=G^{*}$ and $\eta_j=w_j^{*}$ for every $j\in J$, Theorem \ref{th:3.1} gives that $G^{*}$ is hyperbolic.
Now, Lemma \ref{l:2} gives that $G_1\times P_2$ is hyperbolic and we conclude that $G_1\times G_2$ is hyperbolic by Lemma \ref{l:P2}.

%The reverse argument proves that $(1)$ implies $(2)$.

Now suppose $(1)$ holds. By Lemma \ref{l:P2}, $G_1\times P_2$ is hyperbolic and, by Theorem \ref{t:non-hyperbolic}, $G_1$ is hyperbolic. Then, Lemma \ref{l:2} gives that $G^{*}$ is hyperbolic and taking $X=G_1^{*}, DX=G^{*}$ and $\eta_j=w_j^{*}$ for every $j\in J$, by Theorem \ref{th:3.1}, $(2)$ holds.
\end{proof}

Theorem \ref{t:char} and Remark \ref{r:M-regular} have the following consequence.

\begin{corollary}\label{c:char}
Let $G_1$ be a graph and suppose that there are a positive integer $K$ and a vertex $w\in G_1$, such that every odd cycle in $G_1$ intersects the open ball $B:=B_{G_1}(w,K)$.
Let $G_2$ be a non-trivial bounded graph without odd cycles. Then, $G_1\times G_2$ is hyperbolic if and only if $G_1$ is hyperbolic.
\end{corollary}

%------------------------------------SECTION 3

\
  \section{Bounds for the hyperbolicity constant of some direct products}

%%%%%%%%%%%%%%%%%%%%%%%%%%%%%%%%%%%%%%%%%%%%%%%%%%%%%%%%%%%%%%%%%%%%%%%%%%%%%%%%%%%
The following well-known result will be useful (see a proof, e.g., in \cite[Theorem 8]{RSVV}).

\begin{theorem}
\label{t:diameter1}
In any graph $G$ the inequality $\d(G)\le (\diam G )/ 2$ holds.
\end{theorem}

\begin{remark}\label{r:bipartite}
Note that if $G_1$ is a bipartite graph, then $\diam G_1 = \diam V(G_1)$.
Furthermore, if $G_2$ is a bipartite graph, then the product $G_1\times G_2$ has exactly two connected components, which will be denoted by $(G_1\times G_2)^1$ and $(G_1\times G_2)^2$, where each one is a bipartite graph and, consequently, $\diam (G_1\times G_2)^i = \diam V((G_1\times G_2)^i)$ for $i\in \{1,2\}$.
\end{remark}

\begin{remark}\label{r:path}
Let $P_m, P_n$ be two path graphs with $m\ge n\ge 2$. The product $P_m\times P_n$ has exactly two connected components, which will be denoted by $(P_m\times P_n)^1$ and $(P_m\times P_n)^2$.
If $u, v\in V((P_m\times P_n)^i)$ for $i\in \{1,2\}$, then $d_{(P_m\times P_n)^i}(u,v)= \max\big\{d_{P_m}(\pi_1(u),\pi_1(v)),d_{P_n}(\pi_2(u),\pi_2(v)) \big\}$ and $\diam (P_m\times P_n)^i= \diam V((P_m\times P_n)^i)= m-1$.

Furthermore, if $m_1\le m$ and $n_1\le n$ then $\d(P_m\times P_n)\ge \d(P_{m_1}\times P_{n_1})$.
\end{remark}

\begin{lemma}\label{l:path}
Let $P_m, P_n$ be two path graphs with $m\ge n\ge 3$, and let $\g$ be a geodesic in $P_m\times P_n$ such that there are two different vertices $u,v$ in $\g$, with $\pi_1(u)=\pi_1(v)$. Then, $L(\g)\le n-1$.
\end{lemma}

\begin{proof}
Let $\g:=[xy]$, and let $V(P_m)=\{v_1, \ldots, v_m\}, V(P_n)=\{w_1, \ldots, w_n\}$ be the sets of vertices in $P_m, P_n$, respectively, such that $[v_j,v_{j+1}]\in E(P_m)$ and $[w_i,w_{i+1}]\in E(P_n)$ for $1\le j<m, 1\le i<n$.
Seeking for a contradiction, assume that $L(\g) > n-1$.
Notice that if $[uv]$ denotes the geodesic contained in $\g$ joining $u$ and $v$, then $\pi_2$ restricted to $[uv]$ is injective.
Consider two vertices $u',v'\in \g$ such that $[uv]\subseteq [u'v']\subseteq \g$, $\pi_2$ is injective in $[u'v']$ and $\pi_2(u')=w_{i_{1}}$, $\pi_2(v')=w_{i_{2}}$ with $i_2-i_1$ maximal under these conditions.
Since $L(\g) > n-1\geq i_2-i_1$,  either there is an edge $[v',w]$ in $G_1\times G_2$ such that $[v',w]\cap (\g \setminus [u'v'])\neq \emptyset$ or there is an edge $[u',w']$ in $G_1\times G_2$ such that $[u',w']\cap (\g \setminus [u'v'])\neq \emptyset$. Also, since $L(\g) > n-1$, notice that $\pi_2$ is not injective in $\g$. Moreover, since $i_2-i_1$ is maximal, if $\pi_2(w)=w_{i_2+1}$, then $w\notin \g$, and since $L(\g) > n-1$, $u'\notin \{x,y\}$ and $\pi_2(w')=w_{i_1+1}$. Thus, either $\pi_2(w)=w_{i_2-1}$ or $\pi_2(w')=w_{i_1+1}$.

Hence, let us assume that there is an edge $[v',w]$ in $G_1\times G_2$ such that $[v',w]\cap (\g \setminus [u'v'])\neq \emptyset$ with $\pi_2(w)=w_{i_2-1}$ (otherwise, if there is an edge $[u',w']$ in $G_1\times G_2$ such that $[u',w']\cap (\g \setminus [u'v'])\neq \emptyset$ with $\pi_2(w')=w_{i_1+1}$, the proof is similar).

Suppose $\pi_1(v')=v_j$. Let $v''$ be the vertex in $[u'v']$ such that $\pi_2(v'')=w_{i_2-1}$.  Then, by construction of $G_1\times G_2$, since $v''\neq w$, it follows that $\{\pi_1(v''),\pi_1(w)\}=\{v_{j-1},v_{j+1}\}$. Therefore, in particular, $1<j<m$.

Assume that $v''=(v_{j-1},w_{i_2-1})$ (if $v''=(v_{j+1},w_{i_2-1})$, then the argument is similar). Therefore, $w=(v_{j+1},w_{i_2-1})$.

%Since $L(\g) > n-1$, notice that $\pi_2$ is not injective in $\g$. Since $i_2-i_1$ is maximal, either  there is an edge $[u',w]$ in $G_1\times G_2$ such that $[u',w]\cap (\g \setminus [u'v'])\neq \emptyset$ with $\pi_2(w)=i_1+1$ or there is an edge $[v',w]$ in $G_1\times G_2$ such that $[v',w]\cap (\g \setminus [u'v'])\neq \emptyset$ with $\pi_2(w)=i_2-1$.
%Assume that there is an edge $[v',w]$ as above (otherwise, the proof is similar). Notice that this implies that if $\pi_1(v')=v_j$, then $1<j<m$. Assume also that $[(v_{j-1},w_{i_2-1}),v']\subset [u'v']$ (if $[(v_{j+1},w_{i_2-1}),v']\subset [u'v']$, then the argument is similar). Then,
%$w=(v_{j+1},w_{i_2-1})$.

Consider the geodesic
$$\s= [(v_{j+1},w_{i_2-1}),(v_j,w_{i_2-2})]\cup [(v_j,w_{i_2-2}),(v_{j-1},w_{i_2-3})]\cup [(v_{j-1},w_{i_2-3}),(v_{j-2},w_{i_2-4})]\cup \ldots$$

Since $\pi_1(u)=\pi_1(v)$, there is a vertex $\xi$ of $V(P_m\times P_n)$ in $[u'v']\cap \s$.
Let $s\in [v',w]\cap \gamma$ with $s\neq v'$. Let $\s_0$ be the geodesic contained in $\s$ joining $\xi$ and $w$. Let $\g_0$ be the geodesic contained in $\g$ joining $\xi$ and $s$. Hence, $L(\s_0\cup [ws])<L(\s_0)+1 < L(\g_0)$ leading to contradiction.
\end{proof}

\begin{theorem}\label{t:path}
Let $P_m, P_n$ be two path graphs with $m\ge n\ge 2$. If $n=2$, then $\d(P_m\times P_2)=0$. If $n\ge 3$, then
\[\min\Big\{\frac{m}{2}\,,\,n -1\Big\} -1\le\d(P_m\times P_n)\le \min\Big\{\frac{m}{2}\,,\,n \Big\} - \frac{1}{2}\,.\]
Furthermore, if $m\le 2n-3$ and $m$ is odd, then $\d(P_m\times P_n)=(m-1)/2$.
\end{theorem}

\begin{proof}
If $m\ge 2$, then $P_m\times P_2$ has two connected components isomorphic to $P_m$, and $\d(P_m\times P_2)=0$.

Assume that $n\ge 3$.
By symmetry, it suffices to prove the inequalities for $\d((P_m\times P_n)^1)$.
Hence, Theorem \ref{t:diameter1} and Remark \ref{r:path} give $\d((P_m\times P_n)^1)\le \frac{m-1}{2}$.
By Theorem \ref{t:3}, there exists a geodesic triangle $T=\{x,y,z\}\in \mathbb{T}_1(P_m\times P_n)$ with $p\in \g_1:=[xy], \g_2:=[xz], \g_3:=[yz]$, and $\d((P_m\times P_n)^1)=\d(T)=d_{(P_m\times P_n)^1}(p,\g_2\cup \g_3)$. Let $u\in V(\g_1)$ such that $d_{(P_m\times P_n)^1}(p,u)\le 1/2$.

In order to prove $\d((P_m\times P_n)^1)\le n-1/2$, we consider two cases.

Assume first that there is at least a vertex $v\in V((P_m\times P_n)^1)\cap T\setminus \{u\}$ such that $\pi_1(u)=\pi_1(v)$.
If $v\notin \g_1$, then $v\in \g_2\cup \g_3$ and
$$\d(T)= d_{(P_m\times P_n)^1}(p,\g_2\cup \g_3)\le 1/2 + d_{(P_m\times P_n)^1}(u,v)\le n - 1/2.$$
If $v\in \g_1$, then $L(\g_1)\le n-1$ by Lemma \ref{l:path}, and
 $$\d(T)= d_{(P_m\times P_n)^1}(p,\g_2\cup \g_3)\le  d_{(P_m\times P_n)^1}(p,\{x,y\})\le (n-1)/2 < n-1/2.$$

Assume now that there is not a vertex $v\in V((P_m\times P_n)^1)\cap T\setminus \{u\}$ such that $\pi_1(u)=\pi_1(v)$. Then, there exist two different vertices $v_1, v_2$ in $T\setminus \{u\}$ such that $d_{(P_m\times P_n)^1}(u,v_1)= d_{(P_m\times P_n)^1}(u,v_2)=1$, and $\pi_1(v_1)=\pi_1(v_2)$.
If $v_1$ or $v_2$ belongs to $\g_2\cup \g_3$, then $\d(T)=d_{(P_m\times P_n)^1}(p,\g_2\cup \g_3)\le 3/2\le n- 1/2$. Otherwise, $v_1, v_2\in \g_1\setminus \{u\}$.
Lemma \ref{l:path} gives $L(\g_1)\le n-1$, and we have that
$$\d(T)= d_{(P_m\times P_n)^1}(p,\g_2\cup \g_3)\le  d_{(P_m\times P_n)^1}(p,\{x,y\})\le (n-1)/2 < n-1/2.$$

In order to prove the lower bound, denote the vertices of $P_m$ and $P_n$ by
$V(P_m) = \{w_1,w_2,w_3,\ldots,w_m\}$ and $V(P_n) = \{v_1,v_2,v_3,\ldots,v_n\}$, with $[w_i,w_{i+1}]\in E(P_m)$ for $1\le i < m$ and $[v_i,v_{i+1}]\in E(P_n)$ for $1\le i < n$.

Let $(P_m\times P_n)^1$ be the connected component of $P_m\times P_n$ containing $(w_1,v_{n-1})$.

Assume first that $m\ge 2n-3$.
Consider the following curves in $(P_m\times P_n)^1$:
$$
\begin{aligned}
\gamma_1 & :=  [(w_1,v_{n-1}),(w_2,v_n)] \cup [(w_2,v_n),(w_3,v_{n-1})] \cup  [(w_3,v_{n-1}),(w_4,v_n)] \cup  \cdots \cup [(w_{2n-4},v_n),(w_{2n-3},v_{n-1})],
\\
\gamma_2 & :=  [(w_1,v_{n-1}),(w_2,v_{n-2})] \cup [(w_2,v_{n-2}),(w_3,v_{n-3})] \cup \cdots \cup [(w_{n-2},v_2),(w_{n-1},v_1)] \cup [(w_{n-1},v_1),(w_n,v_2)]\\
&\cup  \cdots \cup [(w_{2n-4},v_{n-2}),(w_{2n-3},v_{n-1})].
\end{aligned}
$$
Corollary \ref{c:lower} gives that $\gamma_1, \gamma_2$ are geodesics.
If $B$ is the geodesic bigon $B=\{\g_1,\g_2\}$, then Remark \ref{r:path} gives that
$$
\d(P_m\times P_n)\ge \d(B)\ge d_{(P_m\times P_n)^1}((w_{n-1},v_1),\g_1)=n-2.
$$
\indent
If $m$ is odd with $m\le 2n-3$, then $n-(m+1)/2 \ge 1$ and
we can consider the curves in $(P_m\times P_n)^1$:
$$
\begin{aligned}
\gamma_1 & :=  [(w_1,v_{n-1}),(w_2,v_n)] \cup [(w_2,v_n),(w_3,v_{n-1})] \cup  [(w_3,v_{n-1}),(w_4,v_n)] \cup  \cdots \cup [(w_{m-1},v_n),(w_m,v_{n-1})],
\\
\gamma_2 & :=  [(w_1,v_{n-1}),(w_2,v_{n-2})] \cup [(w_2,v_{n-2}),(w_3,v_{n-3})] \cup \cdots \cup [(w_{(m+1)/2-1},v_{n-(m+1)/2+1}),(w_{(m+1)/2},v_{n-(m+1)/2})] \\
&\cup [(w_{(m+1)/2},v_{n-(m+1)/2}),(w_{(m+1)/2+1},v_{n-(m+1)/2+1})]\cup  \cdots \cup [(w_{m-1},v_{n-2}),(w_m,v_{n-1})].
\end{aligned}
$$
Corollary \ref{c:lower} gives that $\gamma_1, \gamma_2$ are geodesics.
If $B=\{\g_1,\g_2\}$, then
Remark \ref{r:path} gives that
$$\d(P_m\times P_n)\ge \d(B)\ge d_{(P_m\times P_n)^1}((w_{(m+1)/2},v_{n-(m+1)/2}),\g_1)=(m-1)/2.$$

By Remark \ref{r:path}, if $m$ is even with $m-1\le 2n -3$, then we have that
$$\d(P_m\times P_n)\ge \d(P_{m-1}\times P_n)\ge (m-2)/2.$$

Hence,
\[\d(P_m\times P_n)\ge \left\{
\begin{array}{ll}
n-2,\quad &\mbox{if }\ m\ge 2n-3\\
(m-2)/2,\quad &\mbox{if }\ m\le 2n-2
\end{array}
\right\}= \min\Big\{n-2,\frac{m-2}{2}\Big\}=\min\Big\{\frac{m}{2},n-1\Big\}-1.
\]

Furthermore, if $m\le 2n-3$ and $m$ is odd, then we have proved $(m-1)/2\le \d(P_m\times P_n)\le (m-1)/2$.
\end{proof}

\begin{theorem} \label{t:bipartite}
If $G_1$ and $G_2$ are bipartite graphs with $k_1:=\diam V(G_1)$ and $k_2:=\diam V(G_2)$ such that $k_1\ge k_2\ge 1$, then
$$\max\Big\{\min\Big\{\frac{k_1-1}{2}, k_2-1\Big\},\d(G_1),\d(G_2)\Big\}\le \d(G_1\times G_2)\le \frac{k_1}{2}.$$
Furthermore, if $k_1\le 2k_2-2$ and $k_1$ is even, then $\d(G_1\times G_2)=k_1/2$.
\end{theorem}

\begin{proof}
Corollary \ref{c:lower}, Theorem \ref{t:diameter1} and Remark \ref{r:bipartite} give us the upper bound.

In order to prove the lower bound, we can see that there exist two path graphs $P_{k_1+1}, P_{k_2+1}$ which are isometric subgraphs of $G_1$ and $G_2$, respectively.
It is easy to check that $P_{k_1+1}\times P_{k_2+1}$ is an isometric subgraph of $G_1\times G_2$.
By Lemma \ref{l:subgraph} and Theorem \ref{t:path}, we have
$$
\min\Big\{\frac{k_1-1}{2}, k_2-1\Big\}\le \d(P_{k_1+1}\times P_{k_2+1})\le \d(G_1\times G_2).
$$
\indent
Using a similar argument as above, we have $\d(P_2\times G_2)\le \d(G_1\times G_2)$ and $\d(G_1\times P_2)\le \d(G_1\times G_2)$.
Thus, since $(G_1\times P_2)^i \simeq G_1$ and $(P_2\times G_2)^i \simeq G_2$ for $i\in \{1,2\}$, we obtain the first statement.

Furthermore, if $k_1+1\le 2(k_2+1)-3$ and $k_1+1$ is odd, then Theorem \ref{t:path} gives $\d(P_{k_1+1}\times P_{k_2+1})=k_1/2$,
and we conclude $\d(G_1\times G_2)=k_1/2$.
\end{proof}

The following result deals just with odd cycles since otherwise we can apply Theorem \ref{t:bipartite}.

\begin{theorem}\label{t:CmxPn}
For every odd number $m\ge3$ and every $n\ge2$,
\[
\d(C_m\times P_n)=\left\{
\begin{array}{ll}
m/2 ,\quad &\mbox{if }\ n-1\le m,\\
(n-1)/2,\quad &\mbox{if }\ m < n-1 < 2m,\\
m-1/2,\quad &\mbox{if }\ n-1 \ge 2m.
\end{array}
\right.
\]
\end{theorem}

\begin{proof}
Let $V(C_m)=\{w_1,\ldots,w_m\}$ and $V(P_n)=\{v_1,\ldots,v_n\}$ be the sets of vertices in $C_m$ and $P_n$, respectively, such that $[w_1,w_{m}],[w_j,w_{j+1}]\in E(C_m)$ and $[v_i, v_{i+1}]\in E(P_n)$ for $j\in\{1,\ldots,m-1\}$, $i\in\{1,\ldots,n-1\}$.
Note that for $1\le j,r\le m$ and $1\le i,s\le n,$ we have $d_{C_m\times P_n}\big( (w_j,v_i), (w_r,v_s)\big)=\max\{ |i-s|,|j-r| \}$, if $|i-s|\equiv |j-r| (\text{mod } 2)$, or $d_{C_m\times P_n}\big( (w_j,v_i), (w_r,v_s)\big)=\max\{ |i-s|,m-|j-r| \}$, if $|i-s|\not\equiv |j-r| (\text{mod } 2)$.
Besides, we have $\diam(C_m\times P_n)=\diam\big(V(C_m\times P_n)\big)$, i.e., $\diam(C_m\times P_n)=m$ if $n-1\le m$, and $\diam(C_m\times P_n)=n-1$ if $n-1> m$.
Thus, by Theorem \ref{t:diameter1} we have
\[
\d(C_m\times P_n)\le \left\{
\begin{array}{ll}
m/2 ,\quad &\mbox{if }\ n-1\le m,\\
(n-1)/2,\quad &\mbox{if }\ n-1 > m.
\end{array}
\right.
\]

Assume first that $n-1\le m$. Note that $C_m\times P_2 \simeq C_{2m}$ and $C_m\times P_{n^\prime}$ is an isometric subgraph of $C_m\times P_n$, if $n^\prime\le n$.
By Lemma \ref{l:subgraph}, we have $\d(C_m\times P_n)\ge \d(C_{2m})=m/2$, and we obtain the result in this case.

\smallskip

Assume now that $n-1>m$. Consider the geodesic triangle $T$ in $C_m\times P_n$ defined by the following geodesics
\[
\begin{aligned}
\g_1 := &[(w_{1},v_n),(w_{2},v_{n-1})]\cup[(w_{2},v_{n-1}),(w_{3},v_{n})]\cup[(w_{3},v_{n}),(w_{4},v_{n-1})]\cup\ldots\cup[(w_{m-1},v_{n-1}),(w_{m},v_{n})], \\
\g_2 := & [(w_{(m+1)/2},v_1),(w_{(m-1)/2},v_{2})]\cup[(w_{(m-1)/2},v_{2}),(w_{(m-3)/2},v_{3})]\cup\ldots\cup[(w_{2},v_{(m-1)/2}),(w_{1},v_{(m+1)/2})]\cup \\
& [(w_1,v_{(m+1)/2}),(w_m,v_{(m+3)/2})]\cup[(w_m,v_{(m+3)/2}),(w_1,v_{(m+5)/2})]\cup[(w_1,v_{(m+5)/2}),(w_m,v_{(m+7)/2})]\cup\ldots, \\
\g_3 := & [(w_{(m+1)/2},v_1),(w_{(m+3)/2},v_{2})]\cup[(w_{(m+3)/2},v_{2}),(w_{(m+5)/2},v_{3})]\cup\ldots\cup[(w_{m-1},v_{(m-1)/2}),(w_{m},v_{(m+1)/2})]\cup \\
& [(w_m,v_{(m+1)/2}),(w_1,v_{(m+3)/2})]\cup[(w_1,v_{(m+3)/2}),(w_m,v_{(m+5)/2})]\cup[(w_m,v_{(m+5)/2}),(w_1,v_{(m+7)/2})]\cup\ldots ,
\end{aligned}
\]
where $(w_1,v_n)$ \big(respectively, $(w_m,v_n)$\big) is an endpoint of either $\g_2$ or $\g_3$, depending of the parity of $n$.
Since $T$ is a geodesic triangle in $C_m\times P_n$, we have $\d(C_m\times P_n)\ge \d(T)$.
If $n-1 < 2m$ and $M$ is the midpoint of the geodesic $\g_3$,
then $\d(C_m\times P_n)\ge \d(T)= d_{C_m\times P_n}(M,\g_1\cup\g_2)= L(\g_3)/2=(n-1)/2$. Therefore, the result for $m<n-1<2m$ follows.

\smallskip

Finally, assume that $n-1\ge 2m$. Let us consider $N\in\g_3$ such that $d_{C_m\times P_n}\big( N,(w_{(m+1)/2},v_1) \big)=m-1/2$.
Thus, $\d(C_m\times P_n)\ge \d(T) \ge d_{C_m\times P_n}(N, \g_1\cup\g_2)=d_{C_m\times P_n}\big( N,(w_{(m+1)/2},v_1) \big)=m-1/2$.
In order to finish the proof, it suffices to prove that $\d(C_m\times P_n)\le m-1/2$.
Seeking for a contradiction, assume that $\d(C_m\times P_n)>m-1/2$.
By Theorems \ref{l:cuartos} and \ref{t:3}, there is a geodesic triangle $\triangle=\{x,y,z\}\in \mathbb{T}_1(C_m\times P_n)$ and $p\in[xy]$ with 
$d_{C_m\times P_n}(p, [yz]\cup[zx]) = \d(C_m\times P_n)\geq m - 1/4$.
Then, $L([xy])= d_{C_m\times P_n}(x,p) + d_{C_m\times P_n}(p,y) \geq 2m-1/2$.
Let $V_x$ (respectively, $V_y$) be the closest vertex to $x$ (respectively, $y$) in $[xy]$, and consider a vertex $V_p$ in $[xy]$ such that $d_{C_m\times P_n}\big(p, V(C_m\times P_n)\big)=d_{C_m\times P_n}(p, V_p)$.
Note that $d_{C_m\times P_n}(p, [yz]\cup[zx]) \geq m - 1/4$ implies that $d_{C_m\times P_n}(p,V_p) \le 1/2$.
Since $x,y,z \in J(C_m\times P_n)$, we have $d_{C_m\times P_n}(V_x,V_y)\geq 2m-1 > m$ and, consequently, $\pi_{2}([xy])$ is a geodesic in $P_n$.
Since $\pi_{2}([yz]\cup [zx])$ is a path in $P_n$ joining $\pi_{2}(x)$ and $\pi_{2}(y)$, there exists a vertex $(u,v)\in [xz]\cup[zy]$ such that $\pi_{2}(V_p)=v$ and $u\neq\pi_{1}(V_p)$.
Therefore, $d_{C_m\times P_n}\big(V_p,(u,v)\big) \le m-1$ and, consequently, $d_{C_m\times P_n}(p,[xz]\cup[zy])\le d_{C_m\times P_n}(p,V_p) + d_{C_m\times P_n}(V_p,[xz]\cup[zy])\le 1/2 + m-1$, leading to contradiction.
\end{proof}

%%%%%%%%%%%%%%%%%%%%%%%%%%%%%%%%%%%%%%%%%%%%%%%%%%%%%%%%%%%%%%%%%%%%
%%%%%%%%%%%%%%%%%%%%%%%%%%%%%%%%%%%%%%%%%%%%%%%%%%%%%%%%%%%%%%%%%%%%


\begin{thebibliography}{99}

\bibitem{AAD} Abu-Ata, M. and Dragan, F. F., Metric tree-like structures in real-life networks: an empirical study, {\it Networks} {\bf 67} (2016), 49-68.

\bibitem{ASM} Adcock, A. B., Sullivan, B. D. and Mahoney, M. W., Tree-like
structure in large social and information networks, 13th Int
Conference Data Mining (ICDM), IEEE, Dallas,Texas, USA, 2013, pp. 1-10.

\bibitem{ABCD}
Alonso, J., Brady, T., Cooper, D., Delzant, T., Ferlini, V.,
Lustig, M., Mihalik, M., Shapiro, M. and Short, H.,
Notes on word hyperbolic groups,
in: E. Ghys, A. Haefliger, A. Verjovsky (Eds.),
Group Theory from a Geometrical Viewpoint,
World Scientific, Singapore, 1992.

\bibitem{APRT} Alvarez, V., Portilla, A., Rodr\'iguez, J.M. and Tour\'is, E., Gromov hyperbolicity of Denjoy domains,
{\it Geometriae Dedicata} {\bf 121} (2006), 221-245.

\bibitem{BP} Balakrishnan, R. and Paulraja, P., Hamilton cycles in tensor product of graphs,
{\it Discrete Math.} {\bf 186} (1998), 1-13.

\bibitem{BH} Bendall, S. and Hammack, R.,
Centers of $n$-fold tensor products of graphs,
{\it Discuss. Math. Graph Theory} {\bf 24} (2004), 491-501.

\bibitem{BCRS} Bermudo, S., Carballosa, W., Rodr\'{i}guez, J. M. and Sigarreta, J. M.,
On the hyperbolicity of edge-chordal and path-chordal graphs, to appear in FILOMAT.
http://journal.pmf.ni.ac.rs/filomat/filomat/article/view/1090.html

%\bibitem{BRS2} Bermudo, S., Rodr\'{\i}guez, J. M., Rosario, O. and Sigarreta, J. M.,
%Small values of the hyperbolicity constant in graphs.
%Preprint in http://gama.uc3m.es/index.php/jomaro.html

\bibitem{BRS} Bermudo, S., Rodr\'{\i}guez, J. M. and Sigarreta, J. M.,
Computing the hyperbolicity constant,
{\it Comput. Math. Appl.} {\bf 62} (2011), 4592-4595.


%\bibitem{BRST} Bermudo, S., Rodr\'{\i}guez, J. M., Sigarreta, J. M. and Tour{\'\i}s, E.,
%Hyperbolicity and complement of graphs,
%{\it Appl. Math. Letters} {\bf 24} (2011), 1882-1887.

\bibitem{BRSV2} Bermudo, S., Rodr\'{\i}guez, J. M., Sigarreta, J. M. and Vilaire, J.-M.,
Gromov hyperbolic graphs,
{\it Discr. Math.} {\bf 313} (2013), 1575-1585.

\bibitem{BPK} Bogu\~n\'a, M., Papadopoulos, F. and Krioukov, D., Sustaining the Internet with Hyperbolic
Mapping, {\it Nature Commun.} {\bf 1}(62) (2010), 18 p.

\bibitem{Bo} Bonk, M.,
Quasi-geodesics segments and Gromov hyperbolic spaces,
{\it Geom. Dedicata} {\bf 62} (1996), 281-298.

\bibitem{BHB} Bowditch, B. H., Notes on Gromov's hyperbolicity criterion for path-metric spaces. Group theory from a
geometrical viewpoint, Trieste, 1990 (ed. E. Ghys, A. Haefliger and A. Verjovsky; World Scientific, River
Edge, NJ, 1991) 64-167.

\bibitem{BIKB} Bre\v{s}ar, B., Imrich, W., Klav\v{z}ar, S. and Zmazek, B.,
Hypercubes as direct products,
{\it SIAM J. Discrete Math.} {\bf 18} (2005), 778-786.

\bibitem{BKM} Brinkmann, G., Koolen J. and Moulton , V., On the hyperbolicity of chordal
graphs, {\it Ann. Comb.} {\bf 5} (2001), 61-69.

\bibitem{CaFu} Calegari, D. and Fujiwara, K., Counting subgraphs in hyperbolic graphs with symmetry,
{\it J. Math. Soc. Japan} {\bf 67} (2015), 1213-1226.

%\bibitem{CGPR} Cant\'{o}n, A., Granados, A., Pestana, D. and Rodr\'{i}guez, J. M.,
%Gromov hyperbolicity of periodic graphs,
%{\it Bull. Malays. Math. Sci. Soc.} DOI 10.1007/s40840-015-0250-x

%\bibitem{CGPR1} Cant\'{o}n, A., Granados, A., Pestana, D. and Rodr\'{i}guez, J. M.,
%Gromov hyperbolicity of periodic planar graphs,
%{\it Acta Math. Sinica} {\bf 30} (2014), 79-90.

\bibitem{CCCR}  Carballosa, W., Casablanca, R. M., de la Cruz, A. and Rodr\'iguez, J. M., Gromov hyperbolicity in strong product graphs,
{\it Electr. J. Comb.} {\bf 20}(3) (2013), P2.

%\bibitem{CDR} Carballosa, W., de la Cruz, A. and Rodr\'{\i}guez, J. M.,
%Gromov hyperbolicity in lexicographic product graphs.
%Submitted.
%Preprint in http://gama.uc3m.es/index.php/jomaro.html

%\bibitem{CPRS} Carballosa, W., Pestana, D., Rodr\'{\i}guez, J. M. and Sigarreta, J. M.,
%Distortion of the hyperbolicity constant of a graph, {\it Electr. J. Comb.} {\bf 19} (2012), P67.

%\bibitem{CPRS1} Carballosa, W., Portilla, A., Rodr\'{i}guez, J. M. and Sigarreta, J. M.,
%Planarity and hyperbolicity in graphs,
%{\it Graphs Combin.} {\bf 31} (2015), 1311-1324.

%\bibitem{CRRS} Carballosa, W., Rodr\'{i}guez, J. M., Rosario, O. and Sigarreta, J. M., Gromov hyperbolicity of minor graphs, submitted.

\bibitem{CRS1} Carballosa, W., Rodr\'iguez, J.M. and Sigarreta, J. M.,
Hyperbolicity in the corona and join of graphs,
{\it Aequ. Math.} {\bf 89} (2015), 1311-1327.

%\bibitem{CRS} Carballosa, W., Rodr\'{\i}guez, J. M. and Sigarreta, J. M.,
%New inequalities on the hyperbolity constant of line graphs,
%to appear in
%{\it Ars Combinatoria.}
%Preprint in http://gama.uc3m.es/index.php/jomaro.html

%\bibitem{CRSV} Carballosa, W., Rodr\'{\i}guez, J. M., Sigarreta, J. M. and Villeta, M.,
%Gromov hyperbolicity of line graphs,
%{\it Electr. J. Comb.} {\bf 18} (2011), P210.

\bibitem{Cha} Charney, R., Artin groups of finite type are biautomatic, {\it Math. Ann.} {\bf 292} (1992), 671-683.

\bibitem{CYY} Chen, B., Yau, S.-T. and Yeh, Y.-N., Graph homotopy and Graham homotopy,
{\it Discrete Math.} {\bf 241} (2001), 153-170.

%\bibitem{CDEHV} Chepoi, V., Dragan, F. F., Estellon, B., Habib, M. and Vaxes Y.,
%Notes on diameters, centers, and approximating trees of $\d$-hyperbolic geodesic spaces and graphs,
%{\it Electr. Notes Discr. Math.} {\bf 31} (2008), 231-234.

\bibitem{CDV} Chepoi, V., Dragan, F. F. and Vax\`es, Y.,
Core congestion is inherent in hyperbolic networks, Submitted.

%\bibitem{Cizek1994} \v{C}i\v{z}ek, N., Klav\v{z}ar, S., On the chromatic number of the lexicographic product and the Cartesian sum of graphs, {\it Discr. Math.} {\bf 134} (1994), 17-24.

%\bibitem{CCDL} Cohen, N., Coudert, D., Ducoffe, G. and Lancin, A.,
%Applying clique-decomposition for computing Gromov hyperbolicity, submitted.

\bibitem{CoCoLa} Cohen, N., Coudert, D. and Lancin, A., Exact and approximate algorithms for computing the hyperbolicity of large-scale graphs. Rapport de recherche RR-8074, INRIA, September 2012.

\bibitem{CoDu} Coudert, D. and Ducoffe, G.,
On the hyperbolicity of bipartite graphs and intersection graphs. Research Report, INRIA Sophia Antipolis - M\'editerran\'ee; I3S; Universit\'e Nice Sophia
Antipolis; CNRS. 2015, pp.12. $<$hal-01220132$>$

\bibitem{CD} Coudert, D. and Ducoffe, G.,
Recognition of $C_4$-Free and $1/2$-Hyperbolic Graphs,
{\it SIAM J. Discrete Math.} {\bf 28} (2014), 1601-1617.

%\bibitem{DX}
%Der-Fen Liu, D. and Zhu, X.,
%Coloring the Cartesian Sum of Graphs,
%{\it Discr. Math.} {\bf 308} (2008), 5928-5936.

\bibitem{FIV} Fournier, H., Ismail, A. and Vigneron, A.,
Computing the Gromov hyperbolicity of a discrete metric space,
{\it Inform. Process. Letters} {\bf 115} (2015), 576-579.

%\bibitem{K50} Frigerio, R. and Sisto, A., Characterizing hyperbolic spaces and real trees,
%{\it Geom. Dedicata} {\bf 142} (2009), 139-149.
%
%\bibitem{FH} Frucht, R. and Harary, F., On the corona of two graphs, {\it Aequationes Math.} {\bf 4}(3) (1970), 322ï¿½324.

\bibitem{GH} Ghys, E. and de la Harpe, P., Sur les Groupes Hyperboliques d'apr\`es Mikhael Gromov. Progress
in Mathematics 83, Birkh\"auser Boston Inc., Boston, MA, 1990.

%\bibitem{GPPR} Granados, A., Pestana, D., Portilla, A. and Rodr\'{i}guez, J. M.,
%Gromov hyperbolicity in Mycielskian Graphs, submitted.

\bibitem{GJ} Grippo, E. and Jonckheere, E. A.,
Effective resistance criterion for negative curvature: application to congestion control.
In Proceedings of 2016 IEEE Multi-Conference on Systems and Control.

\bibitem{G1} Gromov, M., Hyperbolic groups, in ``Essays in group theory".
Edited by S. M. Gersten, M. S. R. I. Publ. {\bf 8}. Springer, 1987, 75-263.

\bibitem{H1} Hammack, R.,
Minimum cycle bases of direct products of bipartite graphs,
{\it Australas. J. Combin.} {\bf 36} (2006), 213-222.

\bibitem{HIK} Hammack, R., Imrich, W. and Klav\v{z}ar, S., Handbook of product graphs, 2nd
ed., Discrete Mathematics and its Applications Series, CRC Press, 2011.

%\bibitem{H} Harary, F., Graph Theory. Reading, MA: Addison-Wesley, 1994.

%\bibitem{HRS} Hern\'{a}ndez, J. C., Rodr\'{i}guez, J. M. and Sigarreta, J. M.,
%Mathematical properties of the hyperbolicity of circulant networks,
%Advances in Mathematical Physics, Volume 2015 (2015), Article ID 723451, 11 pages.

%\bibitem{HRSTV} Hern\'{a}ndez, J. C., Rodr\'{i}guez, J. M., Sigarreta, J. M., Torres-Nu\~{n}ez, Y. and Villeta, M.,
%Gromov hyperbolicity of regular graphs, to appear in {\it Ars Combin.}

\bibitem{IK00} Imrich, W. and Klav\v{z}ar, S.,
Product graphs: Structure and Recognition, John Wiley \& Sons, New York, 2000.

\bibitem{IR} Imrich, W. and Rall, D. F.,
Finite and infinite hypercubes as direct products,
{\it Australas. J. Combin.} {\bf 36} (2006), 83-90.

\bibitem{IS} Imrich, W. and Stadler, P.,
A prime factor theorem for a generalized direct product,
{\it Discuss. Math. Graph Theory} {\bf 26} (2006), 135-140.

\bibitem{JK} Jha, P. K. and Klav\v{z}ar, S.,
Independence in direct-product graphs,
{\it Ars Combin.} {\bf 50} (1998), 53-63.

%\bibitem{K27} Jonckheere, E. and Lohsoonthorn, P., A hyperbolic geometry approach to
%multipath routing, Proceedings of the 10th Mediterranean Conference
%on Control and Automation (MED 2002), Lisbon, Portugal, July 2002. FA5-1.

\bibitem{K21} Jonckheere, E. A., Contr\^ole du traffic sur les r\'eseaux \`a
g\'eom\'etrie hyperbolique--Vers une th\'eorie g\'eom\'etrique de la s\'ecurit\'e
l'acheminement de l'information, {\it J. Europ. Syst. Autom.} {\bf 8} (2002), 45-60.

\bibitem{K22} Jonckheere, E. A. and Lohsoonthorn, P., Geometry of network security,
{\it Amer. Control Conf.} {\bf ACC} (2004), 111-151.

%\bibitem{K23} Jonckheere, E. A., Lohsoonthorn, P. and Ariaesi, F, Upper bound on scaled
%Gromov-hyperbolic delta, {\it Appl. Math. Comp.} {\bf 192} (2007), 191-204.
%
%\bibitem{K24} Jonckheere, E. A., Lohsoonthorn, P. and Bonahon, F., Scaled Gromov
%hyperbolic graphs, {\it J. Graph Theory} {\bf 2} (2007), 157-180.

\bibitem{KK} Kheddouci, H. and Kouider, M.,
Hamiltonian cycle decomposition of Kronecker product of some cubic graphs by cycles,
{\it J. Combin. Math. Combin. Comput.} {\bf 32} (2000), 3-22.

\bibitem{K56} Koolen, J. H. and Moulton, V., Hyperbolic Bridged
Graphs, {\it Europ. J. Comb.} {\bf 23} (2002), 683-699.

\bibitem{KPKVB} Krioukov, D., Papadopoulos, F., Kitsak, M., Vahdat, A. and Boguñ\'a, M.,
Hyperbolic geometry of complex networks,
{\it Physical Review E} {\bf 82}, 036106 (2010).

%\bibitem{KYR} Kuziak, D., Yero, I. G., Rodr\'{\i}guez-Vel\'azquez, J. A.,
%On the strong metric generators of strong product graphs. Submitted.
%arXiv:1307.4724 [math.CO]

\bibitem{LiTu} Li, S. and Tucci, G. H.,
Traffic Congestion in Expanders, $(p,\d)$-Hyperbolic Spaces and Product of Trees, 2013, arXiv:1303.2952 [math.CO].

\bibitem{MP} Mart\'{\i}nez-P\'erez, A.,
Chordality properties and hyperbolicity on graphs.
 {\it Electr. J. Comb.} {\bf 23}(3) (2016), \#P3.51.

%\bibitem{MRSV} Michel, J., Rodr\'{\i}guez, J. M., Sigarreta, J. M. and Villeta, M.,
%Hyperbolicity and parameters of graphs, {\it Ars Comb.} {\bf 100} (2011), 43-63.

\bibitem{MRSV2} Michel, J., Rodr\'{\i}guez, J. M., Sigarreta, J. M. and Villeta, M.,
Gromov hyperbolicity in Cartesian product graphs, {\it Proc. Indian Acad. Sci. Math. Sci.} {\bf 120} (2010), 1-17.

\bibitem{MoSoVi} Montgolfier, F., Soto, M. and Viennot, L., Treewidth and Hyperbolicity of the Internet, In: 10th IEEE
International Symposium on Network Computing and Applications (NCA), 2011, pp. 25–32.

%\bibitem{NS} Narayan, O. and Saniee, I.,
%Large-scale curvature of networks,
%{\it Physical Review E} {\bf 84}, 066108 (2011).

%\bibitem{Ore} Ore, O., Theory of Graphs, American Mathematical Society, 1962.

\bibitem{O} Oshika, K., Discrete groups, AMS Bookstore, 2002.

\bibitem{Pap} Papasoglu, P., An algorithm detecting hyperbolicity,
in Geometric and computational perspectives on infinite groups, DIMACS - Series in Discrete Mathematics and Theoretical Computer Science
Volume 25, AMS, 1996, pp.193-200.

%\bibitem{PeRSV} Pestana, D., Rodr\'{\i}guez, J. M., Sigarreta, J. M. and Villeta, M.,
%Gromov hyperbolic cubic graphs,
%{\it Central Europ. J. Math.} {\bf 10(3)} (2012), 1141-1151.

%\bibitem{PRST} Portilla, A., Rodr\'{\i}guez, J. M., Sigarreta, J. M. and Tour{\'\i}s, E.,
%Gromov hyperbolic directed graphs, to appear in {\it Acta Math. Appl. Sinica.}
%Preprint in http://gama.uc3m.es/index.php/jomaro.html

%\bibitem{PRSV} Portilla, A., Rodr\'{\i}guez, J. M., Sigarreta, J. M. and Vilaire, J.-M.,
%Gromov hyperbolic tessellation graphs, to appear in {\it Utilitas Math.}
%Preprint in http://gama.uc3m.es/index.php/jomaro.html

%\bibitem{R} Rodr\'{\i}guez, J. M.,
%Characterization of Gromov hyperbolic short graphs,
%{\it Acta Math. Sinica} {\bf 30} (2014), 197-212.

%\bibitem{RS} Rodr\'{i}guez, J. M. and Sigarreta, J. M.,
%Gromov hyperbolicity in infinite circulant graphs, submitted.

%\bibitem{RS1} Rodr\'{i}guez, J. M. and Sigarreta, J. M.,
%On the hyperbolicity of some geometric graphs, submitted.

\bibitem{RSVV} Rodr\'{\i}guez, J. M., Sigarreta, J. M., Vilaire, J.-M.
and Villeta, M., On the hyperbolicity constant in graphs, {\it Discr. Math.} {\bf 311} (2011), 211-219.

\bibitem{RT1} Rodr\'{\i}guez, J. M., Tour{\'\i}s, E.,
Gromov hyperbolicity through decomposition of metric spaces,
{\it Acta Math. Hung.}~{\bf 103} (2004), 53-84.
%
%\bibitem{SU} Scheinerman, E., Ullman, D., Fractional Graph Theory, Wiley-Interscience Series in Discrete Mathematics and Optimization, 1997.

\bibitem{Sha1} Shang, Y., Lack of Gromov-hyperbolicity in colored random networks, {\it Pan-American Math. J.} {\bf 21}(1) (2011), 27-36.

\bibitem{Sha2} Shang, Y., Lack of Gromov-hyperbolicity in small-world networks, {\it Cent. Eur. J. Math.} {\bf 10}(3) (2012), 1152-1158.

\bibitem{Sha3} Shang, Y., Non-hyperbolicity of random graphs with given expected degrees, {\it Stoch. Models} {\bf 29}(4) (2013), 451-462.

%\bibitem{SZ} Shao, Z., Zhang, D., The L(2,1)-labeling on Cartesian sum of graphs,
%{\it Appl. Math. Letters} {\bf 21} (2008), 843-848.

\bibitem{ShTa} Shavitt, Y. and Tankel, T., On internet embedding in hyperbolic
spaces for overlay construction and distance estimation, INFOCOM 2004.

%\bibitem{Si} Sigarreta, J. M.,
%Hyperbolicity in median graphs,
%{\it Proc. Indian Acad. Sci. Math. Sci.} {\bf 123} (2013), 455-467.

\bibitem{T} Tour{\'\i}s, E., Graphs and Gromov hyperbolicity of non-constant negatively curved surfaces,
{\it J. Math. Anal. Appl.} {\bf 380} (2011), 865-881.

\bibitem{VeSu} Verbeek, K. and Suri, S., Metric embeddings, hyperbolic space and social networks. In Proceedings of
the 30th Annual Symposium on Computational Geometry, pp. 501-510, 2014.

\bibitem{W} Weichsel, P. M.,
The Kronecker product of graphs,
{\it Proc. Amer. Math. Soc.} {\bf 13} (1962), 47-52.

\bibitem{WZ} Wu, Y. and Zhang, C.,
Chordality and hyperbolicity of a graph, {\it Electr. J. Comb.} {\bf 18} (2011), P43.

\bibitem{Z} Zhu, X.,
A survey on Hedetniemi's conjecture,
{\it Taiwanese J. Math.} {\bf 2} (1998), 1-24.

\end{thebibliography}
\end{document}